\documentclass[12pt]{article}

\usepackage{amsmath, amsthm, amssymb}
\usepackage[all]{xy}
\usepackage{fullpage}
\usepackage{titlesec}
\usepackage{mathrsfs}
\usepackage{amsbsy}
\usepackage{turnstile}
\usepackage{bbm}
\usepackage{yhmath}
\usepackage{tensor}

\usepackage{setspace}

\usepackage[pdftex,bookmarks=true]{hyperref}

\makeatletter
\newcommand*{\@old@slash}{}\let\@old@slash\slash
\def\slash{\relax\ifmmode\delimiter"502F30E\mathopen{}\else\@old@slash\fi}
\makeatother

\titleformat{\section}{\normalsize\bfseries}{\thesection}{1em}{}
\titleformat{\subsection}{\normalsize\bfseries}{\thesubsection}{1em}{}

\numberwithin{equation}{subsection}

\theoremstyle{plain}

\newtheorem{PropSub}[subsection]{Proposition}
\newtheorem{LemSub}[subsection]{Lemma}
\newtheorem{CorSub}[subsection]{Corollary}
\newtheorem{ThmSub}[subsection]{Theorem}

\newtheorem{PropSubSub}[subsubsection]{Proposition}

\theoremstyle{definition}

\newtheorem{DefSub}[subsection]{Definition}
\newtheorem{ExaSub}[subsection]{Example}
\newtheorem{RemSub}[subsection]{Remark}
\newtheorem{ParSub}[subsection]{}
\newtheorem{AssnSub}[subsection]{Assumption}

\newtheorem{ParSubSub}[subsubsection]{}

\newcommand*{\emptybox}{\leavevmode\hbox{}}

\newdir{ >}{{}*!/-10pt/@{>}}

\DeclareMathAlphabet{\mathpzc}{OT1}{pzc}{m}{it}
\DeclareMathAlphabet{\mathcalligra}{T1}{calligra}{m}{n}

\newcommand{\bref}[1]{\textnormal{\ref{#1}}}
\newcommand{\pbref}[1]{\textnormal{(\ref{#1})}}

\newcommand{\A}{\ensuremath{\mathscr{A}}}
\newcommand{\B}{\ensuremath{\mathscr{B}}}
\newcommand{\C}{\ensuremath{\mathscr{C}}}
\newcommand{\D}{\ensuremath{\mathscr{D}}}
\newcommand{\E}{\ensuremath{\mathscr{E}}}
\newcommand{\F}{\ensuremath{\mathscr{F}}}

\newcommand{\sH}{\ensuremath{\mathscr{H}}}

\newcommand{\J}{\ensuremath{\mathscr{J}}}
\newcommand{\K}{\ensuremath{\mathscr{K}}}

\newcommand{\M}{\ensuremath{\mathscr{M}}}
\newcommand{\N}{\ensuremath{\mathscr{N}}}

\newcommand{\V}{\ensuremath{\mathscr{V}}}

\newcommand{\X}{\ensuremath{\mathscr{X}}}
\newcommand{\Y}{\ensuremath{\mathscr{Y}}}

\newcommand{\uX}{\ensuremath{\underline{\mathscr{X}}}}

\newcommand{\uV}{\ensuremath{\underline{\mathscr{V}}}}

\newcommand{\VCAT}{\ensuremath{\V\textnormal{-\textbf{CAT}}}}

\newcommand{\CC}{\ensuremath{\mathbb{C}}}

\newcommand{\RR}{\ensuremath{\mathbb{R}}}
\newcommand{\SSS}{\ensuremath{\mathbb{S}}}
\newcommand{\TT}{\ensuremath{\mathbb{T}}}
\newcommand{\tTT}{\ensuremath{\widetilde{\mathbb{T}}}}

\newcommand{\ONEONE}{\ensuremath{\mathbbm{1}}}

\newcommand{\tT}{\ensuremath{\widetilde{T}}}

\newcommand{\ob}{\ensuremath{\operatorname{\textnormal{\textsf{Ob}}}}}
\newcommand{\mor}{\ensuremath{\operatorname{\textnormal{\textsf{Mor}}}}}
\newcommand{\Ob}{\ob}
\newcommand{\Mor}{\mor}
\newcommand{\Iso}{\ensuremath{\operatorname{\textnormal{\textsf{Iso}}}}}
\newcommand{\Epi}{\ensuremath{\operatorname{\textnormal{\textsf{Epi}}}}}
\newcommand{\Mono}{\ensuremath{\operatorname{\textnormal{\textsf{Mono}}}}}
\newcommand{\StrMono}{\ensuremath{\operatorname{\textnormal{\textsf{StrMono}}}}}
\newcommand{\StrEpi}{\ensuremath{\operatorname{\textnormal{\textsf{StrEpi}}}}}
\newcommand{\Sub}{\ensuremath{\operatorname{\textnormal{\textsf{Sub}}}}}

\newcommand{\Dense}{\ensuremath{\textnormal{\textsf{Dense}}}}
\newcommand{\ClEmb}{\ensuremath{\textnormal{\textsf{ClEmb}}}}
\newcommand{\DenseEmb}{\ensuremath{\textnormal{\textsf{DenseEmb}}}}
\newcommand{\DenseWrt}[1]{\ensuremath{#1\textnormal{-\textsf{Dense}}}}
\newcommand{\ClEmbWrt}[1]{\ensuremath{#1\textnormal{-\textsf{ClEmb}}}}
\newcommand{\DenseEmbWrt}[1]{\ensuremath{#1\textnormal{-\textsf{DenseEmb}}}}

\newcommand{\SigmaDense}{\DenseWrt{\Sigma}}
\newcommand{\SigmaClEmb}{\ClEmbWrt{\Sigma}}

\newcommand{\Shv}{\ensuremath{\textnormal{\textsf{Shv}}}}

\newcommand{\Set}{\ensuremath{\operatorname{\textnormal{\textbf{Set}}}}}
\newcommand{\Two}{\ensuremath{\operatorname{\textnormal{\textbf{2}}}}}
\newcommand{\SET}{\ensuremath{\operatorname{\textnormal{\textbf{SET}}}}}

\newcommand{\SNorm}{\ensuremath{\operatorname{\textnormal{\textbf{SNorm}}}}}
\newcommand{\Norm}{\ensuremath{\operatorname{\textnormal{\textbf{Norm}}}}}
\newcommand{\Ban}{\ensuremath{\operatorname{\textnormal{\textbf{Ban}}}}}

\newcommand{\Adj}{\ensuremath{\operatorname{\textnormal{\textbf{Adj}}}}}
\newcommand{\Mnd}{\ensuremath{\operatorname{\textnormal{\textbf{Mnd}}}}}
\newcommand{\IdmMnd}{\ensuremath{\operatorname{\textnormal{\textbf{IdmMnd}}}}}
\newcommand{\Refl}{\ensuremath{\operatorname{\textnormal{\textbf{Refl}}}}}

\newcommand{\subs}{\ensuremath{\subseteq}}
\newcommand{\sups}{\ensuremath{\supseteq}}

\newcommand{\lt}{\leqslant}

\newcommand{\op}{\ensuremath{\textnormal{op}}}

\newcommand{\pushoutcorner}{\ar@{}[dr]|(.3)\ulcorner}
\newcommand{\pullbackcorner}{\ar@{}[dr]|(.3)\lrcorner}

\begin{document}

\author{\normalsize  Rory B. B. Lucyshyn-Wright\thanks{The author gratefully acknowledges financial support in the form of an NSERC Postdoctoral Fellowship, as well as earlier, partial financial support in the form of an Ontario Graduate Scholarship.}\let\thefootnote\relax\footnote{Keywords: completion; closure; density; monad; idempotent monad; idempotent core; idempotent approximation; normed vector space; adjunction; reflective subcategory; enriched category; factorization system; orthogonal subcategory; sheaf; sheafification; Lawvere-Tierney topology; monoidal category; closed category}\footnote{2010 Mathematics Subject Classification: 18A20, 18A22, 18A30, 18A32, 18A40, 18B25, 18C15, 18D15, 18D20, 18F10, 18F20, 46B04, 46B10, 46B28, 46M99}
\\
\small University of Ottawa, 585 King Edward Ave., Ottawa, ON, Canada K1N 6N5}

\title{\large \textbf{Completion, closure, and density relative to a monad,\\ with examples in functional analysis and sheaf theory}}

\date{}

\maketitle

\abstract{Given a monad $\TT$ on a suitable enriched category $\B$ equipped with a proper factorization system $(\E,\M)$, we define notions of \textit{$\TT$-completion}, \textit{$\TT$-closure}, and \textit{$\TT$-density}.  We show that not only the familiar notions of completion, closure, and density in normed vector spaces, but also the notions of sheafification, closure, and density with respect to a Lawvere-Tierney topology, are instances of the given abstract notions.  The process of $\TT$-completion is equally the enriched idempotent monad associated to $\TT$ (which we call the \textit{idempotent core of $\TT$}), and we show that it exists as soon as every morphism in $\B$ factors as a $\TT$-dense morphism followed by a $\TT$-closed $\M$-embedding.  The latter hypothesis is satisfied as soon as $\B$ has certain pullbacks as well as wide intersections of $\M$-embeddings.  Hence the resulting theorem on the existence of the idempotent core of an enriched monad entails Fakir's existence result in the non-enriched case, as well as adjoint functor factorization results of Applegate-Tierney and Day.}

\section{Introduction} \label{sec:intro}

Examples of monads abound throughout mathematics, particularly since every adjunction determines one, yet monads play also a seemingly more narrow role as theories of algebraic structure, each monad $\TT$ on a category $\B$ determining a category $\B^\TT$ of $\TT$-algebras.  Working in the context of the theory of categories enriched over a symmetric monoidal closed category $\V$, we show herein that, on the other hand, every monad $\TT$ on any suitable category $\B$ also gives rise to concepts that are seemingly more `topological' in nature, namely, canonical notions of \textit{closure}, \textit{density}, \textit{completeness}, \textit{completion}, and \textit{separatedness} with respect to $\TT$.  As a guiding example, we show that when $\B$ is the category of normed or semi-normed vector spaces and $\TT$ is the monad given by taking the double-dual, the resulting notions with respect to $\TT$ coincide with the familiar notions.  As an example of a very different sort, we show that when $\B$ is an (elementary) topos equipped with a Lawvere-Tierney topology $j$ and we take $\TT$ to be the double-dualization monad for $\Omega_j$, we recover the notions of $j$-closure, $j$-density, $j$-sheaf, $j$-sheafification, and $j$-separatedness.

In more detail, let $\TT$ be a monad on a $\V$-category $\B$ equipped with a proper $\V$-enriched factorization system $(\E,\M)$ (\cite{Lu:EnrFactnSys}), and refer to the morphisms in $\M$ as \textit{embeddings}.  Using techniques of enriched factorization systems and orthogonality, we define notions of \textit{$\TT$-dense morphism}, \textit{$\TT$-closed embedding}, and \textit{$\TT$-complete} object.  Next, we make the following assumption, which is satisfied as soon as $\B$ is complete and well-powered with respect to $\M$:  
\begin{equation}\label{eq:assn}
\begin{minipage}{3in}
\textit{Every morphism in $\B$ factors as a $\TT$-dense morphism followed by a $\TT$-closed embedding}.
\end{minipage}
\end{equation}
Under this assumption, the $\TT$-dense morphisms and $\TT$-closed embeddings constitute an enriched factorization system, and we obtain an operation of \textit{$\TT$-closure} of $\M$-subobjects.  We then show that the full subcategory $\B_{(\TT)}$ of $\B$ consisting of the $\TT$-complete objects is reflective in $\B$, and the resulting idempotent monad $\tTT$ on $\B$ we call the \textit{$\TT$-completion} monad.  The $\TT$-completion $\tT B$ of an object $B \in \B$ is gotten as in the example of normed vector spaces:  We take $\tT B$ to be the $\TT$-closure of the $(\E,\M)$-image of the unit morphism $\eta_B:B \rightarrow TB$.

The resulting $\TT$-completion monad $\tTT$ is an idempotent $\V$-enriched monad on $\B$ that inverts (i.e., sends to isomorphisms) exactly the same morphisms as $\TT$ and so is \textit{the idempotent core\footnote{Terminology suggested by F. W. Lawvere.} of $\TT$}, studied in the context of non-enriched categories by Casacuberta and Frei \cite{CasFrei} (under the name of the \textit{idempotent approximation} of $\TT$) and earlier by Fakir \cite{Fak}, who had shown that the idempotent core of a monad on an ordinary category $\B$ exists as soon as $\B$ is complete and well-powered.  Our construction of the $\TT$-completion monad shows that the $\V$-enriched idempotent core $\tTT$ of $\TT$ exists as soon as the factorization assumption \eqref{eq:assn} is satisfied, and Fakir's result is recovered as a corollary.  Note that our result applies even in the absence of set-indexed limits and so applies, for example, in constructing the $j$-sheafification monad for a Lawvere-Tierney topology $j$ on an arbitrary (elementary) topos.

After recalling some preliminary material on monads and adjunctions in 2-categories (\S \bref{sec:2cat_prelims}), enriched categories (\S \bref{sec:prel_enr}), enriched factorization systems (\S \bref{sec:enr_factn_sys}), and closure operators (\S \bref{sec:cl_ops}), we treat aspects of enriched orthogonal subcategories needed in the sequel (\S \bref{sec:orth_pres_adj}).  Next, we treat the basic theory of the idempotent core of an enriched monad (\S \bref{sec:defs}), as no such treatment exists in the literature; in particular, we consider the univeral property of the idempotent core $\tTT$, we establish several equivalent characterizations of $\tTT$ and of its existence, and we examine the relation of the enriched idempotent core to the ordinary.  Next we show that the completion monad on normed (resp. seminormed) vector spaces is the enriched idempotent core of the double-dualization monad (\S \bref{sec:compl_nvs}); as a corollary, we show that the full subcategory consisting of all Banach spaces is the enriched reflective hull of the space of scalars ($\RR$ or $\CC$).  In \S \bref{sec:cmpl_cl_dens} we then define the notions of $\TT$-density, $\TT$-closure, etc., and we prove our general result on the existence of the idempotent $\V$-core $\tTT$, Theorem \bref{thm:sigma_s_complete_are_sep_and_refl}.  In \S \bref{sec:cl_dens_nvs} we return to the example of normed and seminormed vector spaces and show that the notions of $\TT$-density and $\TT$-closure there coincide with the familiar notions.  In \S \bref{sec:shfn_cl_dens} we treat the example of $j$-sheafification, $j$-closure, and $j$-density for a Lawvere-Tierney topology $j$ on a topos $\X$, and we show that the category of $j$-sheaves is the $\X$-enriched reflective hull of $\Omega_j$ in $\X$; note that Lambek and Rattray had shown in \cite{LamRa:LocShRefl} that the $j$-sheafication monad can be obtained by Fakir's construction.

In treating the basic theory of the idempotent core $\tTT$ of an enriched monad $\TT$ on a $\V$-category $\B$, we show in \bref{thm:charn_idm_approx} that if $F \dashv G:\C \rightarrow \B$ is an arbitrary $\V$-adjunction inducing $\TT$, then $\tTT$ exists if and only if $F \dashv G$ factors as a composite $\V$-adjunction
$$
\xymatrix {
\B \ar@/_0.5pc/[rr]_{K}^(0.4){}^(0.6){}^{\top} & & {\B'} \ar@{_{(}->}@/_0.5pc/[ll]_{J} \ar@/_0.5pc/[rr]_{F'}^(0.4){}^(0.6){}^{\top} & & {\C} \ar@/_0.5pc/[ll]_{G'}
}
$$
with $J$ a $\V$-reflective full subcategory inclusion and $F'$ conservative.  Hence our Theorem \bref{thm:sigma_s_complete_are_sep_and_refl} on the existence of $\tTT$ yields a generalization on results of Applegate and Tierney \cite{AppTie} and Day \cite{Day:AdjFactn} to the effect that any adjunction $F \dashv G$ on a suitable category factors in such a way.  An important point that was not made by these authors is that the resulting reflective subcategory $\B'$ depends only on the \textit{monad} $\TT$ induced by $F \dashv G$.  Cassidy, H\'ebert, and Kelly later proved in the non-enriched context a variant of Day's adjoint-factorization result (\cite[Theorem 3.3]{CHK}), and in the second paragraph of their proof they make use of an instance of what we now call here the ($\TT$-dense, $\TT$-closed embedding)-factorization system, there written as $(\N^\uparrow,\N)$.

We consider also the following refinement of the theory of $\TT$-completeness, $\TT$-density, etc.:  Given an enriched monad $\TT$ on a $\V$-category $\B$ equipped with a proper $\V$-prefactorization-system $(\E,\M)$ (\cite{Lu:EnrFactnSys}), together with a given class $\Sigma$ consisting of morphisms inverted by $T$, we define notions of \textit{$\Sigma$-dense morphism}, \textit{$\Sigma$-closed embedding}, and \textit{$\Sigma$-complete object}, and again assuming that every morphism in $\B$ factors as a $\Sigma$-dense morphism followed by a $\Sigma$-closed embedding, we show that those $\Sigma$-complete objects $B \in \B$ that are also \textit{$\TT$-separated} (meaning that $\eta_B:B \rightarrow TB$ is an embedding) constitute a $\V$-reflective subcategory $\B_{(\TT,\Sigma)}$ of $\B$ \pbref{thm:refl_orth}, so that we obtain an idempotent monad $\tTT_\Sigma$, the \textit{$\TT$-separated $\Sigma$-completion monad}.  This reflectivity is related to Day's result \cite[Corollary 2.3]{Day:AdjFactn}.  The `$\TT$-' rather than `$\Sigma$-' notions are recovered when $\Sigma := T^{-1}(\Iso)$, and in the latter case it is notable that every $\TT$-complete object is necessarily $\TT$-separated, provided $(\E,\M)$-factorizations exist.

The given notions of completeness, closure, and density with respect to an enriched monad $\TT$ and/or a class of morphisms $\Sigma$ were employed in the author's recent Ph.D. thesis \cite{Lu:PhD}, in which they were applied with regard to $R$-module objects in a cartesian closed category (and generalizations thereupon) in providing a basis for abstract functional analysis in a closed category.  In such a context, the usual notions of completeness, closure, and density familiar from functional analysis are not typically available, and so one may instead employ the above `$\TT$-' notions, with respect to the double-dualization monad $\TT$, and these we call \textit{functional completeness}, \textit{functional closure}, and \textit{functional density}.

\section{Preliminaries}\label{sec:prelims}

\subsection{2-categorical preliminaries}\label{sec:2cat_prelims}

\begin{ParSubSub}\label{par:mnd_morphs}
Given an object $\B$ in a 2-category $\K$, there is a category $\Mnd_\K(\B)$ whose objects are monads on $\B$ and whose morphisms $\theta:(T,\eta,\mu) \rightarrow (T',\eta',\mu')$ (called \textit{maps of monads} in \cite{KeStr:RevEl2Cats}) consist of a 2-cell $\theta:T \rightarrow T'$ such that $\theta \cdot \eta = \eta'$ and $\mu' \cdot (\theta \circ \theta) = \theta \cdot \mu$.  The identity monad $\ONEONE_\B$ is an initial object in $\Mnd_\K(\B)$, since for each monad $\TT = (T,\eta,\mu)$ on $\B$, the 2-cell $\eta$ is the unique monad morphism $\eta:\ONEONE_\B \rightarrow \TT$.
\end{ParSubSub}

\begin{ParSubSub}\label{par:talg}
Given a monad $\TT = (T,\eta,\mu)$ on $\B$ in a 2-category $\K$, a \textit{$\TT$-algebra} $(B,\beta)$ in $\K$ (\cite[3.1]{KeStr:RevEl2Cats}) consists of a 1-cell $B:\A \rightarrow \B$ equipped with an \textit{action} of $\TT$ on $B$, i.e. a 2-cell $\beta:TB \rightarrow B$ with $\beta \cdot \eta B = 1_B$ and $\beta \cdot T\beta = \beta \cdot \mu B$.  Given a morphism of monads $\theta:\TT \rightarrow \TT' = (T',\eta',\mu')$, it is shown in \cite[(3.8), (3.9)]{KeStr:RevEl2Cats} that the composite $\beta := (TT' \xrightarrow{\theta T'} T'T' \xrightarrow{\mu'} T')$ is an action of $\TT$ on $T'$ and that $\theta$ can be expressed in terms of $\beta$ as the composite $T \xrightarrow{T\eta'} TT' \xrightarrow{\beta} T'$.
\end{ParSubSub}

\begin{ParSubSub}\label{par:idm_mnd}
Recall that a monad $\SSS = (S,\rho,\lambda)$ on $\B$ in $\K$ is said to be \textit{idempotent} if $\lambda:SS \rightarrow S$ is an isomorphism (from which it then follows that $\lambda^{-1} = S\rho = \rho S$).  If $\SSS$ is idempotent, then for any $\SSS$-algebra $(B,\beta)$ \pbref{par:talg}, the 1-cell $\beta:SB \rightarrow B$ is an isomorphism with inverse $\rho B$.
\end{ParSubSub}

\begin{ParSubSub}\label{par:adj}
Given objects $\A$, $\B$ in a 2-category $\K$, there is a category $\Adj_\K(\A,\B)$ whose objects are adjunctions $F \nsststile{\varepsilon}{\eta} G : \B \rightarrow \A$ in $\K$ and whose morphisms $(\phi,\psi):(F \nsststile{\varepsilon}{\eta} G) \rightarrow (F' \nsststile{\varepsilon'}{\eta'} G')$ consist of 2-cells $\phi:F \rightarrow F'$ and $\psi:G \rightarrow G'$ such that $(\psi \circ \phi) \cdot \eta = \eta'$ and $\varepsilon' \cdot (\phi \circ \psi) = \varepsilon$.

There is a functor $\Adj_{\K}(\A,\B) \rightarrow \Mnd_\K(\A)$ sending an adjunction to its induced monad and a morphism $(\phi,\psi):(F \nsststile{\varepsilon}{\eta} G) \rightarrow (F' \nsststile{\varepsilon'}{\eta'} G')$ to the morphism $\psi \circ \phi:\TT \rightarrow \TT'$ between the induced monads.  Hence, in particular, isomorphic adjunctions induce isomorphic monads.
\end{ParSubSub}

\begin{PropSubSub} \label{thm:mon_func_detd_by_adj}
Let $F \nsststile{\varepsilon}{\eta} G : \B \rightarrow \A$ be an adjunction in a 2-category $\K$.  Then there is an associated monoidal functor $[F,G] := \K(F,G) : \K(\B,\B) \rightarrow \K(\A,\A)$ with the following property:  For any adjunction $F' \nsststile{\varepsilon'}{\eta'} G':\C \rightarrow \B$ with induced monad $\TT'$ on $\B$, the monad $[F,G](\TT')$ on $\A$ is equal to the monad induced by the composite adjunction 
$$\xymatrix{\A \ar@/_0.5pc/[rr]_F^(0.4){\eta}^(0.6){\varepsilon}^{\top} & & \B \ar@/_0.5pc/[ll]_G \ar@/_0.5pc/[rr]_{F'}^(0.4){\eta'}^(0.6){\varepsilon'}^{\top} & & {\C\;.} \ar@/_0.5pc/[ll]_{G'}}$$
\end{PropSubSub}
\begin{proof}
The monoidal structure on the functor $[F,G]$ consists of the morphisms $GH\varepsilon KF:GHFGKF \rightarrow GHKF$ in $\K(\A,\A)$ (for all objects $H$, $K$ in $\K(\B,\B)$) and the morphism $\eta:1_\A \rightarrow GF$ in $\K(\A,\A)$.  The verification is straightforward.
\end{proof}

\begin{PropSubSub} \label{thm:uniq_adj}
Let $F \nsststile{\varepsilon}{\eta} G$ and $F' \nsststile{\varepsilon'}{\eta'} G$ be adjunctions, having the same right adjoint $G:\B \rightarrow \A$, in a 2-category $\K$.  Then these adjunctions are isomorphic and hence induce isomorphic monads on $\A$.
\end{PropSubSub}
\begin{proof}
The 2-cell $\phi := \varepsilon F' \cdot F \eta':F \rightarrow F'$ has inverse $\varepsilon' F \cdot F' \eta$ (\cite{Gray}, I,6.3), and one checks that $(\phi,1_G)$ serves as the needed isomorphism of adjunctions.
\end{proof}

\begin{PropSubSub} \label{thm:mnd_morph_from_adj_factn}
Let $\xymatrix{\A \ar@/_0.5pc/[rr]_F^(0.4){\eta}^(0.6){\varepsilon}^{\top} & & \B \ar@/_0.5pc/[ll]_G \ar@/_0.5pc/[rr]_{F'}^(0.4){\eta'}^(0.6){\varepsilon'}^{\top} & & \C \ar@/_0.5pc/[ll]_{G'}}$ and $\xymatrix{\A \ar@/_0.5pc/[rr]_{F''}^(0.4){\eta''}^(0.6){\varepsilon''}^{\top} & & \C \ar@/_0.5pc/[ll]_{G''}}$ be adjunctions in a 2-category $\K$, with respective induced monads $\TT$, $\TT'$, $\TT''$, and suppose that $GG' = G''$.  Then there is an associated monad morphism $\TT \rightarrow \TT''$.
\end{PropSubSub}
\begin{proof}
Let $F_c \nsststile{\varepsilon_c}{\eta_c} G''$ be the composite adjunction, and let $\TT_c$ be its induced monad.  By \bref{thm:mon_func_detd_by_adj}, we have that $\TT_c = [F,G](\TT')$, whereas $\TT = [F,G](\ONEONE_\B)$.  By applying $[F,G]$ to the monad morphism $\eta':\ONEONE_\B \rightarrow \TT'$, we obtain a monad morphism
$$G \eta' F = [F,G](\eta') : \TT = [F,G](\ONEONE_\B) \rightarrow [F,G](\TT') = \TT_c\;.$$
Also, by \bref{thm:uniq_adj}, there is an isomorphism of monads $\xi:\TT_c \rightarrow \TT''$, and we obtain a composite morphism of monads $\TT \xrightarrow{G \eta' F} \TT_c \xrightarrow{\xi} \TT''$.
\end{proof}

\subsection{Preliminaries on enriched categories}\label{sec:prel_enr}

In what follows, we work in the context of the theory of categories enriched in a  symmetric monoidal category $\V$, as documented in the seminal paper \cite{EiKe} and the comprehensive references \cite{Ke:Ba}, \cite{Dub}.  We shall include an explicit indication of $\V$ when employing notions such as $\V$-category, $\V$-functor, and so on, omitting the prefix $\V$ only when concerned with the corresponding notions for non-enriched or \textit{ordinary} categories.  When such ordinary notions and terminology are applied to a given $\V$-category $\A$, they should be interpreted relative to the underlying ordinary category of $\A$.  In the absence of any indication to the contrary, we will assume throughout that $\V$ is a \textit{closed} symmetric monoidal category, and in this case we denote by $\uV$ the $\V$-category canonically associated to $\V$, whose underlying ordinary category is isomorphic to $\V$; in particular, the internal homs in $\V$ will therefore be denoted by $\uV(V_1,V_2)$.  We do not assume that any limits or colimits exist in $\V$.

\begin{ParSubSub}\label{par:cat_classes}
The ordinary categories $\C$ considered in this paper are \textit{not} assumed \textit{locally small}---that is, they are not necessarily $\Set$-enriched categories.  Rather, we assume that for each category $\C$ under consideration, there is a category $\SET$ of \textit{classes} in which lie the hom-classes of $\C$, so that $\C$ is $\SET$-enriched, but $\SET$ is not assumed cartesian closed.
\end{ParSubSub}

\begin{ParSubSub}\label{par:vmono}
A morphism $m:B_1 \rightarrow B_2$ in a $\V$-category $\B$ (i.e., in the underlying ordinary category of $\B$) is a \textit{$\V$-mono(morphism)} if $\B(A,m):\B(A,B_1) \rightarrow \B(A,B_2)$ is a monomorphism in $\V$ for each object $A \in \B$.  A \textit{$\V$-epi(morphism)} is a $\V$-mono in $\B^\op$.  We denote the classes of all $\V$-monos and $\V$-epis in $\B$ by $\Mono_\V\B$ and $\Epi_\V\B$, respectively.
\end{ParSubSub}

\begin{ParSubSub}\label{par:vlimits}
A \textit{$\V$-limit} in a $\V$-category $\B$ is a \textit{conical} $\V$-enriched limit in the sense of \cite{Ke:Ba} \S 3.8, equivalently a limit (in the underlying ordinary category of $\B$) that is preserved by each (ordinary) functor $\B(A,-):\B \rightarrow \V$ $(A \in \B)$.  In particular, we obtain the notions of $\V$-product, $\V$-fibre-product, etc.  \textit{$\V$-colimits} in $\B$ are $\V$-limits in $\B^\op$.  Any $\V$-fibre-product $m:A \rightarrow B$ of a family of $\V$-monos $(m_i:A_i \rightarrow B)_{i \in I}$ is again a $\V$-mono (\cite{Lu:EnrFactnSys} 2.7); we then say that $m$ is a \textit{$\V$-intersection} of the $m_i$.
\end{ParSubSub}

\begin{PropSubSub} \label{prop:adjn_via_radj_and_unit}
Let $G:\C \rightarrow \B$ be a $\V$-functor, and suppose that for each $B \in \B$ we are given an object $FB$ in $C$ and a morphism $\eta_B:B \rightarrow GFB$ in $\B$ such that for each $C \in \C$, the composite
$$\phi_{B C} := \left(\C(FB,C) \xrightarrow{G_{FB C}} \B(GFB,GC) \xrightarrow{\B(\eta_B,GC)} \B(B,GC)\right)$$
is an isomorphism in $\V$.  Then the given morphisms $\eta_B$ constitute the unit $\eta$ of a $\V$-adjunction \mbox{$F \nsststile{}{\eta} G$} in which $F$ acts in the given way on objects and is given on homs by formula (6) of Ch. 0 of \textnormal{\cite{Dub}}.
\end{PropSubSub}
\begin{proof}
Fixing an object $B \in \B$, the weak Yoneda lemma (\cite{Ke:Ba}, 1.9) yields a bijection between morphisms $B \rightarrow GFB$ and $\V$-natural transformations $\C(FB,-) \rightarrow \B(B,G-)$, under which $\eta_B$ corresponds to $\phi_{B -} := (\phi_{B C})_{C \in \C}$.  In particular, $\phi_{B -}:\C(FB,-) \rightarrow \B(B,G-)$ is thus a $\V$-natural isomorphism, showing that $\B(B,G-):\C \rightarrow \uV$ is a representable $\V$-functor.  Since this holds for all $B \in \B$, the result follows from Proposition 0.2 of \cite{Dub}.
\end{proof}

\begin{ParSubSub}\label{par:refl_subcats_and_idm_mnds}
Let $\B$ be a $\V$-category.  A \textit{$\V$-reflective-subcategory} of $\B$ is a full replete sub-$\V$-category $\B'$ of $\B$ for which the inclusion $\V$-functor $J:\B' \hookrightarrow \B$ has a left $\V$-adjoint $K$.  Any such $\V$-adjunction $K \nsststile{}{\rho} J : \B' \hookrightarrow \B$ is called a \textit{$\V$-reflection} (on $\B$).  Given an idempotent $\V$-monad $\SSS = (S,\rho,\lambda)$ \pbref{par:idm_mnd} on $\B$, we let $\B^{(\SSS)}$ denote the $\V$-reflective-subcategory of $\B$ consisting of those objects $B$ for which $\rho_B$ is iso.
\end{ParSubSub}

\begin{PropSubSub} \label{thm:refl_idemp_mnd}
There is a bijection $\Refl_\V(\B) \cong \IdmMnd_{\VCAT}(\B)$ between the class $\Refl_\V(\B)$ of all $\V$-reflections on $\B$ and the class $\IdmMnd_{\VCAT}(\B)$ of all idempotent $\V$-monads on $\B$, which associates to each $\V$-reflection on $\B$ its induced $\V$-monad.  The $\V$-reflective-subcategory associated to a given idempotent $\V$-monad $\SSS$ via this bijection is $\B^{(\SSS)}$.
\end{PropSubSub}

\subsection{Enriched factorization systems}\label{sec:enr_factn_sys}

Given morphisms $e:A_1 \rightarrow A_2$, $m:B_1 \rightarrow B_2$ in a $\V$-category $\B$, we say that $e$ is \textit{$\V$-orthogonal} to $m$, written $e \downarrow_\V m$, if the commutative square
\begin{equation}\label{eqn:orth_pb}
\xymatrix {
\B(A_2,B_1) \ar[rr]^{\B(A_2,m)} \ar[d]_{\B(e,B_1)}  & & \B(A_2,B_2) \ar[d]^{\B(e,B_2)} \\
\B(A_1,B_1)  \ar[rr]^{\B(A_1,m)}                     & & \B(A_1,B_2)
}
\end{equation}
is a pullback in $\V$.  Given classes $\E$, $\M$ of morphisms in $\B$, we define
associated classes of morphisms as follows:
$$\E^{\downarrow_\V} := \{m\;|\; \forall e \in \E \;:\; e \downarrow_\V m\},\;\;\;\;\M^{\uparrow_\V} := \{e \;|\; \forall m \in \M \;:\; e \downarrow_\V m\}.$$
The pair $(\E,\M)$ is called a \textit{$\V$-prefactorization-system} on $\B$ if condition 1 below holds, and $(\E,\M)$ is called a \textit{$\V$-factorization system} if conditions 1 and 2 both hold:
\begin{enumerate}
\item $\E^{\downarrow_\V} = \M$ and $\M^{\uparrow_\V} = \E$.
\item Each morphism in $\B$ factors as a morphism in $\E$ followed by a morphism in $\M$.
\end{enumerate}
For ordinary categories, with $\V = \SET$ one obtains the familiar notion of \textit{factorization system}; in this case, we drop the indication of $\V$ from the notation.  The relation of $\V$-factorization-systems to ordinary factorization systems is elaborated in \cite{Lu:EnrFactnSys}, where theorems on the existence of $\V$-factorization-systems are proved as well.

\begin{ParSubSub}\label{par:stab_props_prefactn}
Given a $\V$-prefactorization-system $(\E,\M)$ on $\B$, the following stability properties of $\E$ and $\M$ are established in \cite[4.4]{Lu:EnrFactnSys}.  The class $\M$ is closed under composition, cotensors, arbitrary $\V$-fibre-products, and $\V$-pullbacks along arbitrary morphisms in $\B$.  Further, if $g \cdot f \in \M$ and $g \in \M$, then $f \in \M$.  Also, if $\E = \sH^{\uparrow_\V}$ for some class of $\V$-epimorphisms $\sH$, then whenever $g \cdot f \in \M$, it follows that $f \in \M$.  Since $(\M,\E)$ is a $\V$-prefactorization-system in $\B^\op$, one obtains stability properties for $\E$ that are exactly dual to the above properties of $\M$.  Analogous stability properties hold for a prefactorization system on an ordinary category $\B$, even when $\B$ is not locally small; cf. \cite[2.1.1]{FrKe}.
\end{ParSubSub}

\begin{ParSubSub}\label{par:m_subobj_compl}
Given a $\V$-prefactorization-system $(\E,\M)$ on $\B$ with $\M$ a class of $\V$-monos in $\B$, we say that $\B$ is \textit{$\M$-subobject-complete (as a $\V$-category)} if $\B$ is cotensored and has $\V$-intersections \pbref{par:vlimits} of arbitrary (class-indexed) families of $\M$-morphisms, as well as $\V$-pullbacks of $\M$-morphisms along arbitrary morphisms.  By \cite[7.4]{Lu:EnrFactnSys}, if $\B$ is $\M$-subobject-complete, then the following hold:
\begin{enumerate}
\item $(\E,\M)$ is a $\V$-factorization-system on $\B$.
\item For any class $\Sigma$ of morphisms in $\B$, if we let $\N := \Sigma^{\downarrow_\V} \cap \M$, then $(\N^{\uparrow_\V},\N)$ is a $\V$-factorization-system on $\B$.
\end{enumerate}
\end{ParSubSub}

\begin{ParSubSub}\label{par:vproper_prefactn}
A $\V$-prefactorization-system $(\E,\M)$ on $\B$ is said to be \textit{$\V$-proper} if every morphism in $\E$ is a $\V$-epimorphism in $\B$ and every morphism in $\M$ is a $\V$-monomorphism.  A \textit{$\V$-strong-mono(morphism)} in $\B$ is a $\V$-mono to which each $\V$-epi in $\B$ is $\V$-orthogonal, and a \textit{$\V$-strong-epi(morphism)} in $\B$ is a $\V$-strong-mono in $\B^\op$.  We denote the classes of all such by $\StrMono_\V\B$ and $\StrEpi_\V\B$, respectively.  In a tensored and cotensored $\V$-category, these notions reduce (by \cite[6.8]{Lu:EnrFactnSys}) to the familiar notions of strong monomorphism (resp. strong epimorphism), applied to the underlying ordinary category of $\B$.  For any $\V$-proper $\V$-prefactorization-system $(\E,\M)$, we have $\StrMono_\V\B \subseteq \E^{\downarrow_\V} = \M$ and similarly $\StrEpi_\V\B \subseteq \E$; hence, since every section is a $\V$-strong-mono (\cite[6.3]{Lu:EnrFactnSys}), every section therefore lies in $\M$, and dually, every retraction lies in $\E$. 
\end{ParSubSub}

\begin{PropSubSub}
If either of the following conditions holds, then $(\Epi_\V\B,\StrMono_\V\B)$ is a $\V$-factorization-system on $\B$ and $\B$ is $\M$-subobject-complete for $\M = \StrMono_\V\B$.
\begin{enumerate}
\item $\B$ is cotensored, well-powered with respect to $\V$-strong-monos, and has small $\V$-limits and $\V$-cokernel-pairs.
\item $\B$ is cotensored and tensored, well-powered with respect to strong monos, and has small limits.
\end{enumerate}
If either of the following conditions holds, then $(\StrEpi_\V\B,\Mono_\V\B)$ is a $\V$-factorization-system and $\B$ is $\M$-subobject-complete for $\M = \Mono_\V\B$.
\begin{enumerate}
\item[3.] $\B$ is cotensored, well-powered with respect to $\V$-monos, and has small $\V$-limits and $\V$-cokernel pairs.
\item[4.] $\B$ is cotensored and tensored, well-powered, and has small limits.
\end{enumerate}
\end{PropSubSub}
\begin{proof}
In each case, the statement that the pair $(\E,\M)$ in question is a $\V$-factorization-system under the given condition is part of Theorem 7.14 of \cite{Lu:EnrFactnSys}.  In cases 1 and 3, it is clear that $\B$ is $\M$-subobject-complete.  In cases 2 and 4, we deduce that $\B$ is $\M$-subobject-complete by \cite[2.4, 6.8]{Lu:EnrFactnSys}.
\end{proof}

\subsection{Closure operators in categories}\label{sec:cl_ops}

Let $(\E,\M)$ be a prefactorization system on a category $\B$ with $\M \subseteq \Mono\B$.  For each object $B$ of $\B$, denote by $\Sub_\M(B)$ the preordered class of all $\M$-morphisms with codomain $B$.  Suppose that for each $f:A \rightarrow B$ in $\B$ and $m \in \Sub_\M(B)$, the pullback $f^{-1}(m)$ of $m$ along $f$ exists; by \bref{par:stab_props_prefactn}, $f^{-1}(m)$ then lies in $\Sub_\M(A)$.  Under these assumptions, we shall recall some basic results concerning the notion of \textit{idempotent closure operator} defined in \cite{DiGi:Cl} and, in more general settings, in \cite{ThDi:ClOps,Th:ClOpsMInt}.

\begin{ParSub}\label{par:idm_cl_op}
An \textit{idempotent closure operator} on $\M$ in $\B$ is an assignment to each object $B$ of $\B$ a monotone map $\overline{(-)}:\Sub_\M(B) \rightarrow \Sub_\M(B)$ such that
\begin{enumerate}
\item $m \leq \overline{m}$ and $\overline{\overline{m}} \leq \overline{m}$ for each $m \in \M$, and
\item for each $f:A \rightarrow B$ in $\B$ and each $n \in \Sub_\M(B)$, $\overline{f^{-1}(n)} \leq f^{-1}(\overline{n})$.
\end{enumerate}
An $\M$-morphism $m \in \Sub_\M(B)$ is said to be \textit{closed} with respect to $\overline{(-)}$ if $\overline{m} \cong m$ in $\Sub_\M(B)$, whereas $m$ is said to be \textit{dense} if $\overline{m} \cong 1_B$.  By 1, each $\M$-morphism $m:M \rightarrow B$ factors as $M \xrightarrow{d_m} \overline{M} \xrightarrow{\overline{m}} B$ for a unique morphism $d_m$, and by \bref{par:stab_props_prefactn}, $d_m \in \M$.  We say that $\overline{(-)}$ is \textit{weakly hereditary} if for each $m \in \M$, $d_m$ is dense.
\end{ParSub}

\begin{ParSub}\label{par:factn_sys_from_cl_op}
Suppose that $(\E,\M)$-factorizations exist.  For all morphisms $f:A \rightarrow B$ and $m:M \rightarrow A$ with $m \in \M$, denote by $f(m)$ the second factor of the $(\E,\M)$-factorization of the composite $M \xrightarrow{m} A \xrightarrow{f} B$.  Then condition 2 above is equivalent to the following condition:
\begin{enumerate}
\item[$\text{2}'$.] $f(\overline{m}) \leq \overline{f(m)}$.
\end{enumerate}
Given a weakly hereditary idempotent closure operator $\overline{(-)}$ on $\M$ in $\B$, we obtain an associated factorization system $(\Dense,\ClEmb)$ on $\B$, where $\ClEmb$ is the class of all closed $\M$-morphisms and $\Dense$ is the class of all \textit{dense} morphisms, i.e. those $f:A \rightarrow B$ in $\B$ whose \textit{image} $f(1_A)$ is a dense $\M$-morphism.
\end{ParSub}

\begin{ParSub}\label{par:cl_op_assoc_to_factn_sys}
Given a factorization system $\F = (\D,\C)$ with $\C \subseteq \M$, we obtain a weakly hereditary idempotent closure operator $\overline{(-)}^\F$ on $\M$ in $\B$ by defining the closure $\overline{m}^\F$ of each $m \in \M$ to be the second factor of the $(\D,\C)$-factorization of $m$.  Supposing as in \bref{par:factn_sys_from_cl_op} that $(\E,\M)$-factorizations exist, the class of closed $\M$-morphisms (resp. dense morphisms) determined by $\overline{(-)}^\F$ is then equal to $\C$ (resp. $\D$).  In the case that $\F = (\D,\C)$ is the factorization system associated to a given closure operator $\overline{(-)}$ on $\M$ \pbref{par:factn_sys_from_cl_op}, we find that $\overline{m}^\F \cong \overline{m}$ in $\Sub_\M(B)$ for all $m \in \Sub_\M(B)$, $B \in \B$.
\end{ParSub}

\begin{PropSub}\label{thm:clemb_via_densemb}
Let $(\Dense,\ClEmb)$ be the factorization system determined by a weakly hereditary idempotent closure operator $\overline{(-)}$ on $\M$ in $\B$.  Then 
$$\ClEmb = \DenseEmb^\downarrow \cap \M\;,$$
where $\DenseEmb := \M \cap \Dense$.
\end{PropSub}
\begin{proof}
Since $(\Dense,\ClEmb)$ is a prefactorization system and $\DenseEmb \subseteq \Dense$ we know that $\ClEmb = \Dense^\downarrow \subseteq \DenseEmb^\downarrow$ and hence $\ClEmb \subseteq \DenseEmb^\downarrow \cap \M$.  For the converse inclusion, suppose that $m:M \rightarrow B$ lies in $\DenseEmb^\downarrow \cap \M$.  Then, taking the $(\Dense,\ClEmb)$-factorization $M \xrightarrow{d} C \xrightarrow{c} B$ of $m$, we find by \bref{par:stab_props_prefactn} that $d$ lies in $\M$ and hence lies in $\DenseEmb$, so $d \downarrow m$.  Therefore, there is a unique morphism $k$ such that the diagram
$$
\xymatrix{
M \ar@{=}[r] \ar[d]_d  & M \ar[d]^m \\
C \ar[r]_c \ar@{-->}[ur]|k & B
}
$$
commutes, whence $c \cong m$ in $\Sub_\M(B)$, so that $m$ is closed as $c$ is so.
\end{proof}

\section{Orthogonality, adjunctions, and reflections} \label{sec:orth_pres_adj}

The following notion of orthogonality in the enriched context was employed in \cite{Day:AdjFactn}.

\begin{DefSub} \label{def:orth_subcat_sigma}
Let $\B$ be a $\V$-category, let $\Sigma$ be a class of morphisms in $\B$, and let $\C$ be a class of objects in $\B$.
\begin{enumerate}
\item For a morphism $f:A_1 \rightarrow A_2$ in $\B$ and an object $B$ in $\B$, we say that $f$ is \textit{$\V$-orthogonal} to $B$, written $f \bot_\V B$, if \mbox{$\B(f,B):\B(A_2,B) \rightarrow \B(A_1,B)$} is an isomorphism in $\V$.
\item We define $\Sigma^{\bot_\V} := \{C \in \ob\B \:|\; \forall f \in \Sigma \::\: f \bot_\V C\}$.  We let $\B_\Sigma$ be the full sub-$\V$-category of $\B$  whose objects are those in $\Sigma^{\bot_\V}$.
\item We define $\C^{\top_\V} := \{f \in \mor\B \:|\: \forall C \in \C \::\: f \bot_\V C\}$.
\item We say that $(\Sigma,\C)$ is a \textit{$\V$-orthogonal-pair} in $\B$ if $\Sigma^{\bot_\V} = \C$ and $\C^{\top_\V} = \Sigma$.
\item Given a functor $F:\B \rightarrow \C$, we denote by $\Sigma_F$ the class of all morphisms in $\B$ inverted by $F$ (i.e. sent to isomorphisms in $\C$).
\end{enumerate}
\end{DefSub}

\begin{RemSub}
For an ordinary category $\B$, with $\V = \SET$ we obtain the familiar notions of orthogonality \cite{FrKe} and orthogonal pair \cite{CasPesPf}, for which we omit the indication of $\V$ and employ the unadorned symbols $\bot$, $\top$.  Enriched orthogonality clearly implies ordinary orthogonality.
\end{RemSub}

\begin{RemSub}
For any class of morphisms $\Sigma$ in $\B$, $(\Sigma^{\bot_\V\top_\V},\Sigma^{\bot_\V})$ is a $\V$-orthogonal-pair in $\B$.  For any class of objects $\C$ in $\B$, $(\C^{\top_\V},\C^{\top_\V\bot_\V})$ is a $\V$-orthogonal-pair in $\B$.
\end{RemSub}

\begin{RemSub}
If $\B$ has a $\V$-terminal object $1$, then it is easy to show that $f \bot_\V B$ iff $f \downarrow_\V !_B$, where $!_B:B \rightarrow 1$.
\end{RemSub}

\begin{PropSub}\label{thm:clos_props_orth_pairs}
Let $(\Sigma,\C)$ be a $\V$-orthogonal-pair in a $\V$-category $\B$.  Then
\begin{enumerate}
\item $\Sigma$ is closed under tensors in $\B$.  I.e., if $h:A_1 \rightarrow A_2$ in lies in $\Sigma$ and $V$ is an object of $\V$ for which tensors $V \otimes A_1$, $V \otimes A_2$ exist in $\B$, then the induced morphism $V \otimes h:V \otimes A_1 \rightarrow V \otimes A_2$ lies in $\Sigma$.
\item $\C$ is closed under $\V$-enriched weighted limits in $\B$.  I.e., given $\V$-functors $C:\J \rightarrow \B$ and $W:\J \rightarrow \uV$ for which a weighted limit $[W,C]$ in $\B$ exists, if $Cj \in \C$ for all $j \in \J$, then $[W,C] \in \C$.
\end{enumerate}
\end{PropSub}
\begin{proof}
1. For each object $C \in \C$, $\B(V \otimes h,C) \cong \uV(V,\B(h,C))$ in the arrow category of $\V$, so since $\B(h,C)$ is iso, $\B(V \otimes h,C)$ is iso, showing that $V \otimes h \in \C^{\top_\V} = \Sigma$.  2.  For each $h \in \Sigma$, we have an isomorphism
$$\B(h,[W,C]) \cong [\J,\uV](W,\B(h,C-)) = \int_{j \in \J}\uV(Wj,\B(h,C_j))$$
in the arrow category of $\V$, but each $\B(h,C_j)$ is iso and hence $\B(h,[W,C])$ is iso, showing that $[W,C] \in \Sigma^{\bot_\V} = \C$. 
\end{proof}

\begin{PropSub}\label{thm:case_where_ord_orth_impl_enr}
Let $\B$ be a $\V$-category, $\Sigma \subseteq \Mor\B$, $\C \subseteq \ob\B$.
\begin{enumerate}
\item If $\B$ is tensored and $\Sigma$ is closed under tensors in $\B$, then $\Sigma^{\bot_\V} = \Sigma^{\bot}$.
\item If $\B$ is cotensored and $\C$ is closed under cotensors in $\B$, then $\C^{\top_\V} = \C^{\top}$.
\end{enumerate}
\end{PropSub}
\begin{proof}
1.  Given $C \in \Sigma^\bot$ and $h \in \Sigma$, it suffices to show that $h \bot_\V C$.  Letting $\B_0$ denote the underlying ordinary category of $\B$, we have that for each $V \in \V$, $\V(V,\B(h,C)) \cong \B_0(V \otimes h,C)$ in the arrow category of $\SET$, and the latter morphism is iso.  Hence $\B(h,C)$ is an isomorphism in $\V$.  2 is proved analogously.
\end{proof}

\begin{PropSub}\label{thm:orth_to_all_is_iso}
For a $\V$-category $\B$, $(\ob\B)^{\top_\V} = \Iso \B = (\mor\B)^{\uparrow_\V}$.
\end{PropSub}
\begin{proof}
$\Iso\B$ is clearly included in both the rightmost and leftmost classes.  Also, the inclusion $(\mor\B)^{\uparrow_\V} \subs \Iso\B$ follows from \cite[3.7]{Lu:EnrFactnSys}.  Lastly, if $h:B_1 \rightarrow B_2$ lies in $(\ob\B)^{\top_\V}$, then the $\V$-natural transformation $\B(h,-):\B(B_2,-) \rightarrow \B(B_1,-)$ is an isomorphism; but by the weak Yoneda lemma (\cite[1.9]{Ke:Ba}), the (ordinary) functor \hbox{$Y:\B^\op \rightarrow \VCAT(\B,\uV)$} given by $YB = \B(B,-)$ is fully-faithful, so $h$ is iso.
\end{proof}

\begin{PropSub} \label{thm:orth_of_morphs_in_orth_subcat}
Let $\B$ be a $\V$-category and $\Sigma$ a class of morphisms in $\B$.  Then for any morphisms $e:A_1 \rightarrow A_2$ in $\Sigma$ and $m:B_1 \rightarrow B_2$ in $\B_\Sigma$, we have that $e \downarrow_\V m$ in $\B$.
\end{PropSub}
\begin{proof}
Since $e \bot_\V B_1$ and $e \bot_\V B_2$, the left and right sides of the commutative square \eqref{eqn:orth_pb} are isomorphisms, so the square is a pullback.
\end{proof}

In the non-enriched context, the first of the following equivalences appears in the proof of Lemma 4.2.1 of \cite{FrKe}, and variants of both equivalences are given in \cite{Pum}.

\begin{PropSub}\label{prop:orth_and_adj}
Let $F \dashv G:\C \rightarrow \B$ be a $\V$-adjunction, $f:B_1 \rightarrow B_2$ a morphism in $\B$, $g:C_1 \rightarrow C_2$ a morphism in $\C$, and $C$ an object of $\C$.
\begin{enumerate}
\item $Ff \downarrow_\V g \;\;\Longleftrightarrow\;\; f \downarrow_\V Gg$.
\item $Ff \bot_\V C \;\;\Longleftrightarrow\;\; f \bot_\V GC$.
\end{enumerate}
\end{PropSub}
\begin{proof}
1.  Via the given $\V$-adjunction, the commutative diagram
$$
\xymatrix {
\C(FB_2,C_1) \ar[rr]^{\C(FB_2,g)} \ar[d]_{\C(Ff,C_1)} & & \C(FB_2,C_2) \ar[d]^{\C(Ff,C_2)} \\
\C(FB_1,C_1) \ar[rr]^{\C(FB_1,g)}                     & & \C(FB_1,C_2)
}
$$
is isomorphic to the commutative diagram
$$
\xymatrix {
\B(B_2,GC_1) \ar[rr]^{\B(B_2,Gg)} \ar[d]_{\B(f,GC_1)} & & \B(B_2,GC_2) \ar[d]^{\B(f,GC_2)} \\
\B(B_1,GC_1) \ar[rr]^{\B(B_1,Gg)}                     & & {\B(B_1,GC_2)\;.}
}
$$
2.  $\C(Ff,C) \cong \B(f,GC)$ in the arrow category $[\Two,\V]$.
\end{proof}

In the non-enriched setting, the first of the following equations is noted in \cite[3.3]{CHK} and specializes \cite[4.2.1]{FrKe}.

\begin{CorSub} \label{thm:charns_of_sigma_p}
Let $F \dashv G:\C \rightarrow \B$ be a $\V$-adjunction.  Then 
$$(G(\mor\C))^{\uparrow_\V} = \Sigma_F = (G(\ob\C))^{\top_\V}\;.$$
Hence, in particular, $(\Sigma_F,\Sigma_F^{\downarrow_\V})$ is a $\V$-prefactorization-system on $\B$, and $(\Sigma_F,\Sigma_F^{\bot_\V})$ is a $\V$-orthogonal-pair in $\B$.
\end{CorSub}
\begin{proof}
By \bref{prop:orth_and_adj} and \bref{thm:orth_to_all_is_iso}, we may compute as follows:
\begin{description}
\item[] $(G(\mor\C))^{\uparrow_\V} = F^{-1}((\mor\C)^{\uparrow_\V}) = F^{-1}(\Iso\C) = \Sigma_F$;  
\item[] $(G(\ob\C))^{\top_\V} = F^{-1}((\ob\C)^{\top_\V}) = F^{-1}(\Iso\C) = \Sigma_F$.
\end{description}
\end{proof}

Clearly any sub-$\V$-category of $\B$ of the form $\B_\Sigma$ is replete.  The following proposition shows that every $\V$-reflective-subcategory of $\B$ is of the form $\B_\Sigma$ for each of two canonical choices of $\Sigma$:

\begin{PropSub} \label{thm:refl_subcats_are_orth_subcats}
Let $K \nsststile{}{\rho} J : \C \hookrightarrow \B$ be a $\V$-reflection.
\begin{enumerate}
\item $\C = \B_{\Sigma_K} = \B_{\Sigma}$, where $\Sigma := \{\rho_B\;|\;B \in \B\}$.
\item $(\mor\C)^{\uparrow_\V} = \Sigma_K = (\ob\C)^{\top_\V}$, where the right- and leftmost expressions are evaluated with respect to $\B$.
\item $(\Sigma_K,\Sigma_K^{\downarrow_\V})$ is a $\V$-prefactorization-system on $\B$, and $(\Sigma_K,\C)$ is a $\V$-orthogonal-pair in $\B$.
\end{enumerate}
\end{PropSub}
\begin{proof}
2 follows from \bref{thm:charns_of_sigma_p}.  Regarding 1, first observe that $\Sigma \subs \Sigma_K$, so that $\B_{\Sigma_K} \subs \B_\Sigma$.  Also, by 2, $\Ob\B_{\Sigma_K} = (\Sigma_K)^{\bot_\V} = (\Ob\C)^{\top_\V\bot_\V} \sups \Ob\C$.  Hence it now suffices to show $\B_\Sigma \subs \C$.  Suppose $B \in \B_\Sigma$.  Since we also have that $KB \in \C \subs \B_\Sigma$, the morphism $\rho_B:B \rightarrow KB$ lies in $\B_\Sigma$, so since $\rho_B \in \Sigma$ we deduce by \bref{thm:orth_of_morphs_in_orth_subcat} that $\rho_B \downarrow_\V \rho_B$.  Hence by \cite[3.7]{Lu:EnrFactnSys}, $\rho_B$ is iso, so $B \in \C$.  Lastly, observe that 3 follows from \bref{thm:charns_of_sigma_p} and 1.
\end{proof}

\begin{PropSub}\label{thm:ord_refl_closed_under_cot_is_enr}
Let $\C$ be a reflective subcategory of the underlying ordinary category of a $\V$-category $\B$, and suppose that $\B$ is cotensored and $\C$ is closed under cotensors in $\B$.  Then $\C$ is a $\V$-reflective-subcategory of $\B$.
\end{PropSub}
\begin{proof}
For all $B \in \B$, the reflection morphism $\rho_B:B \rightarrow KB$ lies in $\C^\top$, but by \bref{thm:case_where_ord_orth_impl_enr}, $\C^\top = \C^{\top_\V}$, so $\rho_B \bot_\V C$ for all $C \in \C$ and the result follows by \bref{prop:adjn_via_radj_and_unit}.
\end{proof}

\begin{DefSub}
Given a full sub-$\V$-category $\C$ of a $\V$-category $\B$, the \textit{$\V$-reflective hull} of $\C$ (in $\B$), if it exists, is the smallest $\V$-reflective-subcategory of $\B$ containing $\C$.
\end{DefSub}

\begin{PropSub}\label{thm:double_perp_refl_hull}
Let $\C$ be a full sub-$\V$-category a $\V$-category $\B$.
\begin{enumerate}
\item Any $\V$-reflective-subcategory of $\B$ that contains $\C$ must also contain $\C^{\top_\V\bot_\V}$.
\item Hence, if $\C^{\top_\V\bot_\V}$ is a $\V$-reflective-subcategory of $\B$, then the $\V$-reflective-hull of $\C$ in $\B$ exists and equals $\C^{\top_\V\bot_\V}$.
\end{enumerate}
\end{PropSub}
\begin{proof}
If $\C \hookrightarrow \D \hookrightarrow \B$ and latter inclusion has left $\V$-adjoint $K:\B \rightarrow \D$, then $\C^{\top_\V\bot_\V} \subseteq \D^{\top_\V\bot_\V} = \Sigma_K^{\bot_\V} = \D$ by \bref{thm:refl_subcats_are_orth_subcats} 2 \& 1.
\end{proof}

\begin{CorSub}\label{thm:sigma_perp_refl_impl_reflhull_imageq}
Let $F \dashv G:\C \rightarrow \B$ be a $\V$-adjunction.  Then if $\B_{\Sigma_F} \hookrightarrow \B$ is a $\V$-reflective-subcategory, it follows that $\V$-reflective-hull of $G(\ob\C) \hookrightarrow \B$ exists and equals $\B_{\Sigma_F}$.
\end{CorSub}
\begin{proof}
$(G(\ob\C))^{\top_\V\bot_\V} = \Sigma_F^{\bot_\V} = \B_{\Sigma_F}$ by \bref{thm:charns_of_sigma_p}, so the result follows from \bref{thm:double_perp_refl_hull} 2. 
\end{proof}

\section{Definition and characterizations of the idempotent core}\label{sec:defs}

\begin{LemSub}\label{thm:morph_on_idm_mnd}
Let $\SSS = (S,\rho,\lambda)$ and $\TT = (T,\eta,\mu)$ be monads on an object $\B$ of \mbox{2-category} $\K$, and suppose that $\SSS$ is idempotent.  Then
\begin{enumerate}
\item A 2-cell $\alpha:S \rightarrow T$ is a morphism of monads $\SSS \rightarrow \TT$ if and only if $\alpha \cdot \rho = \eta$.
\item If a morphism of monads $\SSS \rightarrow \TT$ exists, then it is unique.
\end{enumerate}
\end{LemSub}
\begin{proof}
One of the implications in 1 is trivial; for the other, suppose that $\alpha \cdot \rho = \eta$.  We must show that $\mu \cdot (\alpha \circ \alpha) = \alpha \cdot \lambda$.  But $\lambda$ is an isomorphism with $\lambda^{-1} = \rho S$, and $\mu \cdot (\alpha \circ \alpha) \cdot \rho S = \mu \cdot \alpha T \cdot S \alpha \cdot \rho S = \mu \cdot \alpha T \cdot \rho T \cdot \alpha = \mu \cdot \eta T \cdot \alpha = \alpha$.  Regarding 2, let $\alpha:\SSS \rightarrow \TT$ be a morphism of monads.  By \bref{par:talg}, the 1-cell $T$ carries the structure of an $\SSS$-algebra $(T,\beta)$ where $\beta$ is the composite $ST \xrightarrow{\alpha T} TT \xrightarrow{\mu} T$.  But by \bref{par:idm_mnd}, since $\SSS$ is idempotent, $\beta$ is an isomorphism with inverse $\rho T:T \rightarrow ST$.  Hence by \bref{par:talg}, $\alpha = \beta \cdot S\eta = (\rho T)^{-1} \cdot S\eta$, so we have expressed $\alpha$ in terms of $\SSS$ and $\TT$, showing that $\alpha$ is the unique monad morphism $\SSS \rightarrow \TT$.
\end{proof}

\begin{DefSub}
Let $\TT$ be a $\V$-monad on a $\V$-category $\B$.
\begin{enumerate}
\item If $\tTT$ is an idempotent $\V$-monad on $\B$ for which there exists a (necessarily unique, \bref{thm:morph_on_idm_mnd}) morphism $\iota_\TT:\tTT \rightarrow \TT$ satisfying the following condition, then we say that $\tTT$ is a \textit{terminal idempotent $\V$-monad over $\TT$}:
\begin{enumerate}
\item[] For each morphism of $\V$-monads $\alpha:\SSS \rightarrow \TT$ with $\SSS$ idempotent, there is a unique morphism $\alpha_\sharp:\SSS \rightarrow \tTT$ with $\iota_\TT \cdot \alpha_\sharp = \alpha$.
\end{enumerate}
\item If $\tTT$ is an idempotent $\V$-monad on $\B$ whose underlying endofunctor $\tT$ inverts the same morphisms as $T$, then we say that $\tTT$ is an \textit{idempotent ($\V$-)core of $\TT$}.
\end{enumerate}
\end{DefSub}

\begin{RemSub}
A terminal idempotent $\V$-monad over $\TT$ is equally a terminal object in the category of idempotent $\V$-monads over $\TT$ and so is unique, up to isomorphism, if it exists.  We will see in \bref{thm:charn_idm_approx} that any idempotent core of $\TT$ is in particular a terminal idempotent $\V$-monad over $\TT$.  Hence, an idempotent core of $\TT$ is unique, up to isomorphism, if it exists, in which case any terminal idempotent $\V$-monad over $\TT$ is an idempotent core.
\end{RemSub}

\begin{RemSub}
If for every $\V$-monad $\TT$ on $\B$, the terminal idempotent $\V$-monad over $\TT$ exists, then the full subcategory $\IdmMnd_{\VCAT}(\B)$ of $\Mnd_{\VCAT}(\B)$ consisting of all idempotent $\V$-monads is a coreflective subcategory.
\end{RemSub}

\begin{LemSub} \label{thm:adj_factn_lemma}
Suppose given a $\V$-adjunction $F \nsststile{\varepsilon}{\eta} G : \C \rightarrow \B$ and a $\V$-reflection $K \nsststile{}{\rho} J : \B' \hookrightarrow \B$ such that the image of $G$ lies in $\B'$.  Then there is a $\V$-adjunction $F' \nsststile{\varepsilon'}{\eta'} G':\C \rightarrow \B'$ with $JG' = G$, $F'K \cong F$, $F' = FJ$, $J\eta' = \eta J$, and $\varepsilon' = \varepsilon$.  Hence by \bref{thm:uniq_adj}, the adjunction $F \nsststile{\varepsilon}{\eta} G$ is isomorphic to the composite of the $\V$-adjunctions $K \nsststile{}{\rho} J$ and $F' \nsststile{\varepsilon'}{\eta'} G'$.
\end{LemSub}
\begin{proof}
$G'$ is just the corestriction of $G$, the components of $\eta'$ are just those of $\eta$; the $\V$-naturality of $\eta'$ is immediate, and the triangular equations are readily verified.
\end{proof}

\begin{DefSub} \label{def:adj_factors_through_refl}
Given data as in \bref{thm:adj_factn_lemma}, we say that $F \nsststile{\varepsilon}{\eta} G$ \textit{factors through $\B'$}, and, equivalently, that \textit{$F \nsststile{\varepsilon}{\eta} G$ factors through $K \nsststile{}{\rho} J$}.
\end{DefSub}

\begin{PropSub}\label{thm:charns_adj_factn_via_mnd}
Let $\TT$ be a $\V$-monad, and let $F \nsststile{\varepsilon}{\eta} G : \C \rightarrow \B$ be any $\V$-adjunction inducing $\TT$.  Let $\SSS$ be an idempotent $\V$-monad on $\B$, with associated $\V$-reflection $K \nsststile{}{\rho} J : \B' \hookrightarrow \B$.  Then the following are equivalent:
\begin{enumerate}
\item There exists a \textnormal{(}necessarily unique, \bref{thm:morph_on_idm_mnd}\textnormal{)} morphism of $\V$-monads $\alpha:\SSS \rightarrow \TT$.
\item $\B'$ contains each object $TB$ (with $B \in \B$).
\item $F \nsststile{\varepsilon}{\eta} G$ factors through $K \nsststile{}{\rho} J$.
\end{enumerate}
\end{PropSub}
\begin{proof}
Observe that 2 is equivalent to the statement that the Kleisli $\V$-adjunction for $\TT$ factors through $K \nsststile{}{\rho} J$.  Hence it suffices to prove that 1 $\Leftrightarrow$ 3, for then the equivalence 1 $\Leftrightarrow$ 2 follows as a special case.  If 3 holds, then the existence of a morphism of $\V$-monads $\alpha:\SSS \rightarrow \TT$ is guaranteed by \bref{thm:mnd_morph_from_adj_factn}.  For the converse implication, let us assume 1 and prove that 3 holds.  Working with only the underlying ordinary monad morphism and adjunction, note that the given adjunction determines a comparison functor $\C \rightarrow \B^\TT$, and we have also a functor $\B^\alpha:\B^\TT \rightarrow \B^\SSS$ induced by $\alpha$.  Both these functors commute with the forgetful functors to $\B$, and so too does their composite $\C \rightarrow \B^\TT \rightarrow \B^\SSS \cong \B'$.  Hence, applying this composite functor to any given $C \in \C$, we find that the carrier $GC$ of the associated $\SSS$-algebra lies in $\B'$, so 3 holds.  
\end{proof}

\begin{RemSub}
Proposition \bref{thm:charns_adj_factn_via_mnd} shows in particular that the question of whether a  $\V$-adjunction $F \nsststile{\varepsilon}{\eta} G$ factors through a given $\V$-reflection depends only on the $\V$-monad $\TT$ induced by $F \nsststile{\varepsilon}{\eta} G$.
\end{RemSub}

\begin{CorSub}\label{thm:charn_comps_morph_on_idm_mnd}
Suppose $\SSS = (S,\rho,\lambda)$ and $\TT = (T,\eta,\mu)$ are $\V$-monads on a $\V$-category $\B$ with $\SSS$ idempotent, and let $\alpha:\SSS \rightarrow \TT$ be a morphism of $\V$-monads.  Then for each $B \in \B$, $\alpha_B:SB \rightarrow TB$ is the unique morphism such that $\alpha_B \cdot \rho_B = \eta_B$.
\end{CorSub}
\begin{proof}
By \bref{thm:charns_adj_factn_via_mnd}, we know that $TB$ lies in the reflective subcategory $\B' \hookrightarrow \B$ determined by $\SSS$, so $\alpha_B$ is the unique extension of $\eta_B:B \rightarrow TB$ along the reflection unit component $\rho_B:B \rightarrow SB$.
\end{proof}

\begin{ParSub} \label{def:cats_idmmnd_refl_defn}
Given a $\V$-category $\B$, the class $\Refl_\V(\B)$ \pbref{thm:refl_idemp_mnd} of all $\V$-reflections on $\B$ acquires the structure of a preordered class when ordered by inclusion of the associated $\V$-reflective-subcategories.
\end{ParSub}

\begin{CorSub} \label{thm:refls_for_t_iso_to_idm_mnds_over_q}
Let $\B$ be a $\V$-category.
\begin{enumerate}
\item The full subcategory $\IdmMnd_{\VCAT}(\B)$ of $\Mnd_{\VCAT}(\B)$ is a preordered class isomorphic to $(\Refl_\V(\B))^\op$ via the bijection given in \bref{thm:refl_idemp_mnd}.
\item Given a $\V$-monad $\TT$ on $\B$, the isomorphism in 1 restricts to an isomorphism between the full subcategories determined by the following objects:
\begin{enumerate}
\item Idempotent $\V$-monads $\SSS$ on $\B$ for which a \textnormal{(}necessarily unique, \bref{thm:morph_on_idm_mnd}\textnormal{)} morphism of $\V$-monads $\alpha:\SSS \rightarrow \TT$ exists.
\item $\V$-reflections on $\B$ whose associated $\V$-reflective-subcategory contains each object $TB$ ($B \in \B$).
\end{enumerate}
\end{enumerate}
\end{CorSub}
\begin{proof}
We shall prove 1, and then 2 follows by \bref{thm:charns_adj_factn_via_mnd}.  By \bref{thm:morph_on_idm_mnd}, $\IdmMnd_\V(\B)$ is a preorder, and it suffices to show that the bijection $\Refl_\V(\B) \rightarrow \IdmMnd_\V(\B)$ \pbref{thm:refl_idemp_mnd} and its inverse are contravariantly functorial (i.e. order-reversing).  But this follows from \bref{thm:charns_adj_factn_via_mnd}, since the given preorder relation on $\Refl_\V(\B)$ may equally be described as
$$(K' \nsststile{}{\rho'} J') \lt (K \nsststile{}{\rho} J)\;\;\Longleftrightarrow\;\;\text{$K'\nsststile{}{\rho'} J'$ factors through $K \nsststile{}{\rho} J$}\;.$$
\end{proof}

\begin{ThmSub}\label{thm:charn_idm_corefl}
Let $\TT$ be a $\V$-monad on a $\V$-category $\B$, and let $F \dashv G:\C \rightarrow \B$ be any $\V$-adjunction inducing $\TT$.  Then the following are equivalent:
\begin{enumerate}
\item The terminal idempotent $\V$-monad over $\TT$ exists.
\item The full sub-$\V$-category $T(\ob\B) \hookrightarrow \B$ has a $\V$-reflective hull.
\item The full sub-$\V$-category $G(\ob\C) \hookrightarrow \B$ has a $\V$-reflective hull.
\item There is a smallest $\V$-reflective-subcategory through which $F \dashv G$ factors.
\end{enumerate}
\end{ThmSub}
Further, given an idempotent $\V$-monad $\tTT$ on $\B$ with associated $\V$-reflective-subcategory $\B'$, $\tTT$ is a terminal idempotent $\V$-monad over $\TT$ if and only if $\B'$ is a $\V$-reflective hull of $T(\ob\B)$, resp. $G(\ob\C)$, equivalently, a smallest $\V$-reflective-subcategory through which $F \dashv G$ factors.
\begin{proof}
A terminal idempotent $\V$-monad over $\TT$ is by definition a terminal object of the preordered class described in \bref{thm:refls_for_t_iso_to_idm_mnds_over_q} 2(a), so the result follows from \bref{thm:refls_for_t_iso_to_idm_mnds_over_q} 2 and \bref{thm:charns_adj_factn_via_mnd}.
\end{proof}

\begin{LemSub}\label{thm:all_ladjs_ind_given_monad_inv_same_morphs}
Given an (ordinary) adjunction $F \nsststile{}{\eta} G:\C \rightarrow \B$ with induced endofunctor $T$ on $\B$, $\Sigma_F = \Sigma_T$.  Hence all left adjoints inducing a given monad $\TT$ invert the same morphisms.
\end{LemSub}
\begin{proof}
One inclusion is immediate.  For the other, suppose $f:B \rightarrow B'$ in $\B$ is inverted by $T$.  Then $Tf$ has an inverse $(Tf)^{-1}:TB' \rightarrow TB$, and one easily shows that the transpose $FB' \rightarrow FB$ of the composite $B' \xrightarrow{\eta_{B'}} GFB' \xrightarrow{(Tf)^{-1}} GFB$ under the given adjunction serves as inverse for $Ff$.
\end{proof}

\begin{CorSub}\label{thm:morphs_inv_by_t_vorth_pair}
Let $\TT$ be a $\V$-monad on a $\V$-category $\B$.  Then $(\Sigma_T,\Sigma_T^{\downarrow_\V})$ is a $\V$-prefactorization-system on $\B$, and $(\Sigma_T,\Sigma_T^{\bot_\V})$ is a $\V$-orthogonal-pair in $\B$.
\end{CorSub}
\begin{proof}
Taking any $\V$-adjunction $F \dashv G$ inducing $\TT$ (e.g., the Kleisli $\V$-adjunction), we have that $\Sigma_T = \Sigma_F$ by \bref{thm:all_ladjs_ind_given_monad_inv_same_morphs}, and the result follows from \bref{thm:charns_of_sigma_p}.
\end{proof}

\begin{ThmSub}\label{thm:charn_idm_approx}
Let $\TT$ be a $\V$-monad on a $\V$-category $\B$, and let $F \dashv G:\C \rightarrow \B$ be any $\V$-adjunction inducing $\TT$.  Then the following are equivalent:
\begin{enumerate}
\item The idempotent core of $\TT$ exists.
\item The terminal idempotent $\V$-monad over $\TT$ exists, and its underlying endofunctor inverts the same morphisms as $T$.
\item $\B_{\Sigma_T}(= \Sigma_T^{\bot_\V})$ is a $\V$-reflective-subcategory of $\B$.
\item $F \dashv G$ factors through a $\V$-reflection $K \dashv J:\B' \hookrightarrow \B$ in such a way that the induced left $\V$-adjoint $F':\B' \rightarrow \C$ \pbref{def:adj_factors_through_refl} is conservative (i.e. reflects isomorphisms).
\end{enumerate}
Further, given an idempotent $\V$-monad $\tTT$ on $\B$ with associated $\V$-reflective-subcategory $\B'$, the following are equivalent: (i) $\tTT$ is an idempotent core of $\TT$, (ii) $\tTT$ is a terminal idempotent $\V$-monad over $\TT$ and inverts the same morphisms as $T$, (iii) $\B' = \B_{\Sigma_T}$, (iv) $F \dashv G$ factors through $\B' \hookrightarrow \B$ in such a way that the induced left $\V$-adjoint $F'$ is conservative.
\end{ThmSub}
\begin{proof}
We prove the equivalence of (i)-(iv), from which the equivalence of 1-4 follows.  Let $K \dashv J:\B' \rightarrow \B$ be the $\V$-reflection determined by $\tTT$.  The implication (ii) $\Rightarrow$ (i) is immediate.  To see that (i) $\Rightarrow$ (iii), observe that if (i) holds then $\Sigma_T = \Sigma_{\tT} = \Sigma_K$, and so $\B_{\Sigma_T} = \B_{\Sigma_K} = \B'$ by \bref{thm:refl_subcats_are_orth_subcats} 1.  Regarding the implication (iv) $\Rightarrow$ (iii), we reason that if (iv) holds then we have that $\Sigma_K = \Sigma_{F'K}$ since $F'$ is conservative, but $F'K \cong F$ by \bref{thm:adj_factn_lemma} and $\Sigma_F = \Sigma_T$ by \bref{thm:all_ladjs_ind_given_monad_inv_same_morphs}; therefore $\Sigma_K = \Sigma_{F'K} = \Sigma_F = \Sigma_T$, and using \bref{thm:refl_subcats_are_orth_subcats} 1 we deduce that $\B' = \B_{\Sigma_K} = \B_{\Sigma_T}$.  It now suffices to prove the implications (iii) $\Rightarrow$ (ii) and (iii) $\Rightarrow$ (iv).  Assuming (iii), we have that $\B' = \B_{\Sigma_T} = \B_{\Sigma_F} = \Sigma_F^{\bot_\V}$, since $\Sigma_T = \Sigma_F$ by \bref{thm:all_ladjs_ind_given_monad_inv_same_morphs}.  Hence since $\B'$ is a $\V$-reflective-subcategory of $\B$ we deduce by \bref{thm:sigma_perp_refl_impl_reflhull_imageq} that the $\V$-reflective-hull of $G(\ob\C)$ exists and equals $\B'$, so by \bref{thm:charn_idm_corefl}, $\tTT$ is a terminal idempotent $\V$-monad over $\TT$ and $F \dashv G$ factors through $K \dashv J$.  Using \bref{thm:refl_subcats_are_orth_subcats} 1, we know that $\Sigma_K^{\bot_\V} = \B' = \Sigma_F^{\bot_\V}$, so since $(\Sigma_K,\Sigma_K^{\bot_\V})$ and $(\Sigma_F,\Sigma_F^{\bot_\V})$ are $\V$-orthogonal-pairs by \bref{thm:charns_of_sigma_p}, we deduce that $\Sigma_{\tT} = \Sigma_K = \Sigma_F = \Sigma_T$.  But from this it follows also that $F':\B' \rightarrow \C$ is conservative, since if $F'f$ is iso (for some morphism $f$ in $\B'$), then since $F' = FJ$ \pbref{thm:adj_factn_lemma} and $\Sigma_F = \Sigma_K$ we find that $KJf$ is iso, but $KJ \cong 1_{\B'}$ and hence $f$ is iso.
\end{proof}

\begin{PropSub}\label{prop:underlying_ord_mnd_of_idm_approx_is_idm_approx}
Let $\tTT$ be an idempotent $\V$-core of a $\V$-monad $\TT$ on $\B$.  Then the underlying ordinary monad of $\tTT$ is an idempotent core of the underlying ordinary monad of $\TT$.  Hence, whereas in general, $\V$-orthogonality implies ordinary orthogonality, we have in this case an equation $\Sigma_T^{\bot_\V} = \Sigma_T^{\bot}$.
\end{PropSub}
\begin{proof}
The underlying ordinary monad of $\tTT$ is an idempotent monad which inverts the same morphisms as $\TT$ and hence is an ordinary idempotent core of $\TT$.  Its associated reflective subcategory is $\Sigma_T^{\bot}$ and yet has same objects as the $\V$-reflective-subcategory determined by $\tTT$, which is $\Sigma_T^{\bot_\V}$.
\end{proof}

\begin{ThmSub}\label{thm:tens_cot_ord_existence_suffices}
Let $\TT$ be a $\V$-monad on a tensored and cotensored $\V$-category, and suppose that the idempotent core of the underlying ordinary monad of $\TT$ exists.  Then the idempotent $\V$-core of $\TT$ exists.
\end{ThmSub}
\begin{proof}
By \bref{thm:morphs_inv_by_t_vorth_pair} and \bref{thm:clos_props_orth_pairs}, $\Sigma_T$ is closed under tensors in $\B$, so by \bref{thm:case_where_ord_orth_impl_enr}, $\Sigma_T^{\bot_\V} = \Sigma_T^\bot$.  The latter is a reflective subcategory of $\B$ (by an application of \bref{thm:charn_idm_approx} to the underlying ordinary monad of $\TT$).  But $\Sigma_T^{\bot_\V}$ is closed under cotensors by \bref{thm:clos_props_orth_pairs} and hence by \bref{thm:ord_refl_closed_under_cot_is_enr} is a $\V$-reflective-subcategory of $\B$, so the result follows by \bref{thm:charn_idm_approx}.
\end{proof}

\begin{ExaSub}[Double-dualization monads]\label{exa:dbl_dln}
Let $\V$ be a symmetric monoidal closed category and $R \in \V$ an object.  For each object $V$ of $\V$, we shall call the internal hom $V^* := \uV(V,R)$ the \textit{dual} of $V$ with respect to $R$.  We obtain a $\V$-adjunction
$$
\xymatrix {
\uV \ar@/_0.5pc/[rr]_{(-)^*}^{\top} & & {\uV^\op} \ar@/_0.5pc/[ll]_{(-)^*}
}
$$
which we call \textit{the dualization $\V$-adjunction} (for $R$); it is an instance of a \textit{`hom-cotensor'} $\V$-adjunction \cite[(3.42)]{Ke:Ba}.  We call the induced $\V$-monad $\TT$ (on $\uV$) the \textit{double-dualization $\V$-monad}; its underlying $\V$-functor $(-)^{**}$ sends each $V \in \V$ to the \textit{double-dual} $V^{**}$ of $V$.  Since $\uV$ is tensored and cotensored, \bref{thm:tens_cot_ord_existence_suffices} entails that the idempotent $\V$-core $\tTT$ of $\TT$ exists as soon as the idempotent core of the underlying ordinary monad of $\TT$ exists.  In this case, the $\V$-reflective-subcategory $\B'$ determined by $\tTT$ is the $\V$-reflective-hull of the single object $R$ in $\uV$.  Indeed, by \bref{thm:charn_idm_corefl}, $\B'$ is the $\V$-reflective-hull of $\{\uV(V,R) \;|\; V \in \V\}$ in $\uV$, but any $\V$-reflective-subcategory of $\uV$ containing $R$ is closed under cotensors in $\uV$ and hence contains each cotensor $\uV(V,R)$ of $R$.
\end{ExaSub}

\begin{ExaSub}[Completion of normed vector spaces]\label{exa:compl_nvs}
We shall show in \bref{sec:compl_nvs} that the double-dualization $\V$-monad $\TT$ on the category $\V$ of normed or seminormed vector spaces (over $R = \RR$ or $\CC$) has an idempotent $\V$-core $\tTT$ that associates to each (semi-)normed vector space $V$ the \textit{(Cauchy-)completion} of $V$.  The $\V$-reflective-subcategory of $\uV$ determined by $\tTT$ is the category of Banach spaces, which therefore is the $\V$-reflective-hull of $R$ in $\V$ (by \bref{exa:dbl_dln}).
\end{ExaSub}

\begin{ExaSub}[Sheafification for a Lawvere-Tierney topology]
Given an (elementary) topos $\X$ and a Lawvere-Tierney topology $j$ on $\X$, let $\Omega_j$ be the associated retract of the subobject classifier $\Omega$, and let $\TT$ be the double-dualization $\V$-monad for $\Omega_j$.  We show in \bref{thm:lt_top} that the the idempotent $\V$-core $\tTT$ of $\TT$ is the \textit{$j$-sheafification} $\X$-monad, whose associated $\X$-reflective-subcategory of $\uX$ consists of the $j$-sheaves.  Hence the $\X$-category of $j$-sheaves is by \bref{exa:dbl_dln} the $\X$-reflective-hull of $\Omega_j$ in $\uX$.
\end{ExaSub}

\section{Example: Completion of normed vector spaces}\label{sec:compl_nvs}

Let $\SNorm_1$ be the category of seminormed vector spaces over $R = \RR$ or $\CC$ with nonexpansive linear maps (i.e. bounded linear maps of seminorm $\lt 1$), and let $\Norm_1$ be the full subcategory consisting of normed vector spaces.

Letting $\V$ be either $\SNorm_1$ or $\Norm_1$, it is well-known that $\V$ is symmetric monoidal closed; e.g. see \cite{BoDay} \S 3.4.  Indeed, in both categories, the internal hom $\uV(V,W)$ $(V,W \in \V)$ is the vector space of all bounded linear maps $V \rightarrow W$, equipped with the usual \textit{operator (semi)norm}.  Given seminormed (resp. normed) spaces $V,W$, the monoidal product $V \otimes W$ in $\SNorm_1$ (resp. $\Norm_1$) is the algebraic tensor product, equipped with the \textit{projective} seminorm (resp. norm) $||x|| = \inf\{ \sum_i ||v_i||||w_i|| \;|\; x = \sum_{i = 1}^n v_i \otimes w_i\}$.  (In particular, the projective seminorm is a norm as soon as $V$ and $W$ are normed \cite[Ch. III, Exercise 20]{Sch:Tvs}.)

The \textit{dualization $\V$-functor} $(-)^*:\uV^\op \rightarrow \uV$ \pbref{exa:dbl_dln} associates to each $V \in \V$ the space $V^*$ of all bounded linear functionals on $V$, and assigns to each morphism $h:V_1 \rightarrow V_2$ in $\V$ the map $h:V_2^* \rightarrow V_1^*$ given by $\psi \mapsto \psi \cdot h$.  The \textit{double-dualization $\V$-monad} $\TT$ \pbref{exa:dbl_dln} on $\uV$ associates to each $V \in \V$ its \textit{double-dual} $TV = V^{**}$, and the unit morphism $\eta_V:V \rightarrow TV$ is the familiar canonical linear map, which is always \textit{isometric} (i.e. $||\eta_V(v)|| = ||v||$ for all $v \in V$), so that $\eta_V$ is an \textit{isometric embedding} (i.e., isometric and injective) as soon as $V$ is normed.

In the present section, we show that the double-dualization $\V$-monad $\TT$ on $\uV$ has an idempotent $\V$-core $\tTT$, given by \textit{completion}.

\begin{ParSub}\label{par:cauchy_compl}
The full subcategory $\Ban_1$ of $\V( = \SNorm_1$ or $\Norm_1)$ consisting of all Banach spaces is a $\V$-reflective-subcategory of $\uV$.  Indeed, for an arbitrary morphism \mbox{$\rho_V:V \rightarrow \widetilde{V}$} in $\V$ with $\widetilde{V}$ a Banach space, the following conditions are equivalent, and they characterize (up to isomorphism) the familiar \textit{completion} $\widetilde{V}$ of $V$:
\begin{enumerate}
\item $\rho_V:V \rightarrow \widetilde{V}$ is dense and isometric.
\item For each Banach space $B$, the morphism $\uV(\rho_V,B):\uV(\widetilde{V},B) \rightarrow \uV(V,B)$ in $\V$ is an (isometric) isomorphism.
\item For each morphism $f:V \rightarrow B$ in $\V$ with $B$ a Banach space, there is a unique morphism $f^\sharp:\widetilde{V} \rightarrow B$ in $\V$ with $f^\sharp \cdot \rho_V = f$; further, $||f^\sharp|| = ||f||$.
\end{enumerate}
Concretely, we can take $\widetilde{V}$ to be the familiar \textit{Cauchy-completion} of $V$, or the closure of the image of $\eta_V:V \rightarrow TV = V^{**}$.
\end{ParSub}

\begin{LemSub}\label{thm:dense_iff_dual_injective}
Let $h:V \rightarrow W$ in $\V(=\SNorm_1$ or $\Norm_1)$.  Then $h$ is dense if and only if $h^*:W^* \rightarrow V^*$ is injective.
\end{LemSub}
\begin{proof}
The `only if' part is straightforward.  Conversely, suppose $h$ is not dense.  Then, letting $C \subseteq W$ be the closure of the image of $h$, the Hahn-Banach Extension Theorem entails that there is some nonzero $\psi \in W^*$ with $\psi|_C = 0$.  But then $h^*(\psi) = \psi \cdot h = 0$ and yet $\psi \neq 0$, showing that $h^*$ is not injective.
\end{proof}

\begin{PropSub}\label{thm:densisom_inv_compl_inv_dualn}
For each morphism $h:V \rightarrow W$ in $\V(=\SNorm_1$ or $\Norm_1)$, the following are equivalent:
\begin{enumerate}
\item $h$ is dense and isometric.
\item $\widetilde{h}:\widetilde{V} \rightarrow \widetilde{W}$ is an (isometric) 
isomorphism.
\item $h^*:W^* \rightarrow V^*$ is an (isometric) isomorphism.
\end{enumerate}
\end{PropSub}
\begin{proof}
To show 1 $\Rightarrow$ 2, suppose that $h$ is dense and isometric.  We have a commutative square
\begin{equation}\label{eqn:nat_sq_c}
\xymatrix{
V \ar[d]_h \ar[r]^{\rho_{V}} & \widetilde{V} \ar[d]^{\widetilde{h}} \\
W \ar[r]^{\rho_{W}}          & {\widetilde{W}}
}
\end{equation}
in which both $h$ and $\rho_W$ are dense and isometric, so that the composite $\rho_W \cdot h : V \rightarrow \widetilde{W}$ is a dense, isometric morphism into a Banach space and hence satisfies the universal property characterizing the completion of $V$ \pbref{par:cauchy_compl} .  Using this and also the (same) universal property of $\rho_V$, it follows that $\widetilde{h}$ is an isomorphism.

To show 2 $\Rightarrow$ 3, suppose that $\widetilde{h}$ is an isomorphism.  Then, applying $(-)^*$ to the commutative square \eqref{eqn:nat_sq_c}, we obtain a commutative square
$$
\xymatrix{
V^* & \widetilde{V}^* \ar[l]_{\rho_{V}^*} \\
W^* \ar[u]^{h^*} & {\widetilde{W}^*} \ar[l]_{\rho_{W}^*} \ar[u]_{\widetilde{h}^*}
}
$$
in which the right side is an isomorphism.  But the second characterization of the completion in \bref{par:cauchy_compl} entails that the top and bottom faces are isomorphisms as well, so $h^*$ is an isomorphism.

To show 3 $\Rightarrow$ 1, suppose that $h^*$ is an isomorphism. By \bref{thm:dense_iff_dual_injective} we deduce that $h$ is dense.  Further, we have a commutative square 
$$
\xymatrix{
V \ar[d]_h \ar[r]^{\eta_{V}} & V^{**} \ar[d]^{h^{**}} \\
W \ar[r]^{\eta_{W}}          & W^{**}
}
$$
in which the right side is an isomorphism and the top and bottom sides are isometric, and it follows that $h$ is isometric.
\end{proof}

\begin{ParSub}
Let $\tTT$ be the idempotent $\V$-monad induced by the completion $\V$-adjunction \mbox{$\widetilde{(-)} \nsststile{}{\rho} J:\Ban_1 \hookrightarrow \uV$}, where $\V =\SNorm_1$ or $\Norm_1$, so that $\tT V = \widetilde{V}$ is the completion of $V \in \V$.
\end{ParSub}

\begin{ThmSub}\label{thm:compl_nvs}
The completion $\V$-monad $\tTT$ on the category $\V$ of normed (resp. seminormed) vector spaces is the idempotent $\V$-core of the double-dualization $\V$-monad $\TT$ on $\uV$.  Moreover, $\Sigma_T = \Sigma_{\tT}$ is the class of all dense, isometric morphisms in $\V$.
\end{ThmSub}
\begin{proof}
By \bref{thm:all_ladjs_ind_given_monad_inv_same_morphs}, $\Sigma_T = \Sigma_{(-)^*}$, and by \bref{thm:densisom_inv_compl_inv_dualn}, the latter class equals $\Sigma_{\tT}$ and consists of exactly the dense, isometric morphisms.
\end{proof}

\begin{CorSub}
The category $\Ban_1$ of Banach spaces is the $\V$-reflective-hull of $R$ $(=\RR$ or $\CC$) in the category $\V$ of normed (resp. seminormed) vector spaces.  Moreover, a normed (resp. seminormed) vector space $B$ is a Banach space if and only if the following equivalent conditions hold for each morphism $h:V \rightarrow W$ in $\V$:
\begin{enumerate}
\item If $h \bot_\V R$, then $h \bot_\V B$.
\item If $h^*:W^* \rightarrow V^*$ is an isomorphism (in \V), then $\uV(h,B):\uV(W,B) \rightarrow \uV(V,B)$ is an isomorphism.
\end{enumerate}
\end{CorSub}

\section{Completion, closure, and density relative to a monad. Existence results via factorization.} \label{sec:cmpl_cl_dens}

\begin{ParSub}[Given data] \label{par:data_enr_orth_subcat_subord_adj}
In the present section, we work with given data as follows, which we shall later suppose to satisfy Assumption \bref{assn:dense_closed_factn_sys} below:
\begin{enumerate}
\item Let $\TT = (T,\eta,\mu)$ be a $\V$-monad on a $\V$-category $\B$.
\item Let $\Sigma \subs \Sigma_T$ be a class of morphisms inverted by $T$.
\item Let $(\E,\M)$ be a $\V$-proper $\V$-prefactorization-system on $\B$.
\end{enumerate}
We shall refer to the morphisms in $\M$ as \textit{$\M$-embeddings}.
\end{ParSub}

\begin{DefSub} \label{def:compl_sep_closed_dense}
Given data as in \bref{par:data_enr_orth_subcat_subord_adj}, we make the following definitions:
\begin{enumerate}
\item An object $B$ of $\B$ is \textit{$\Sigma$-complete} if $B \in \B_\Sigma$.
\item An object $B$ of $\B$ is \textit{$\TT$-separated} if $\eta_B:B \rightarrow TB$ is an $\M$-embedding.
\item We denote the full sub-$\V$-category of $\B$ consisting of the $\Sigma$-complete $\TT$-separated objects by $\B_{(\TT,\Sigma)}$.
\item We say that an $\M$-embedding $m:B_1 \rightarrowtail B_2$ in $\B$ is \textit{$\Sigma$-closed} if $m \in \Sigma^{\downarrow_\V}$.  We denote by $\SigmaClEmb := \Sigma^{\downarrow_\V} \cap \M$ the class of all $\Sigma$-closed $\M$-embeddings in $\B$.
\item We say that a morphism in $\B$ is \textit{$\Sigma$-dense} if it lies in $\SigmaDense := \SigmaClEmb^{\uparrow_\V}$.
\item If $\Sigma = \Sigma_T$, then we replace the prefixes ``$\Sigma$-'' with ``$\TT$-'', obtaining the notions of \textit{$\TT$-complete} object, \textit{$\TT$-closed $\M$-embedding}, and \textit{$\TT$-dense} morphism.
\end{enumerate}
\end{DefSub}

\begin{RemSub}\label{rem:sigmadense_sigma_clemb_vprefs}
Observe that $(\SigmaDense,\SigmaClEmb)$ is a $\V$-prefactorization-system on $\B$, since $\SigmaClEmb = \Sigma^{\downarrow_\V} \cap \E^{\downarrow_\V} = (\Sigma \cup \E)^{\downarrow_\V}$.
\end{RemSub}

\begin{AssnSub}\label{assn:dense_closed_factn_sys}
For the remainder of \S \bref{sec:cmpl_cl_dens}, we shall assume that \textit{every morphism in $\B$ factors as a $\Sigma$-dense morphism followed by a $\Sigma$-closed $\M$-embedding}.
\end{AssnSub}

It then follows by \bref{rem:sigmadense_sigma_clemb_vprefs} that $(\SigmaDense,\SigmaClEmb)$ is a $\V$-factorization-system on $\B$.

\begin{ExaSub}\label{exa:msubobjcpl_impl_factn_assn}
Given data as in \bref{par:data_enr_orth_subcat_subord_adj}, Assumption \bref{assn:dense_closed_factn_sys} is satisfied as soon as $\B$ is \mbox{\textit{$\M$-subobject-complete}} \pbref{par:m_subobj_compl}.
\end{ExaSub}

\begin{ExaSub}[Normed and seminormed spaces]
As we shall show in \bref{thm:dens_cl_nvs}, data satisfying Assumption \bref{assn:dense_closed_factn_sys} can be obtained as follows.  Take $\V := \Norm_1$ or $\SNorm_1$ to be the category of normed or seminormed vector spaces \pbref{sec:compl_nvs}, let $\B := \uV$, let $\E$ consist of all surjective morphisms in $\V$, and let $\M$ consist of all isometric embeddings in $\V$.  Let $\TT$ be the double-dualization $\V$-monad on $\uV$ \pbref{sec:compl_nvs}, and let $\Sigma := \Sigma_T$.  Then
\begin{enumerate}
\item a (semi)normed space is $\TT$-complete if and only if it is a Banach space;
\item a seminormed space is $\TT$-separated if and only if it is normed;
\item the $\TT$-closed $\M$-embeddings are exactly the closed isometric embeddings;
\item the $\TT$-dense morphisms are exactly the dense morphisms in the usual sense;
\item $\Sigma_T$ consists of exactly the dense isometric morphisms in $\V$.
\end{enumerate}
\end{ExaSub}

\begin{ExaSub}[Sheaves, density, and closed subobjects]
Given an (elementary) topos $\X$ and a Lawvere-Tierney topology $j$ on $\X$, we show in \bref{thm:lt_top} that data satisfying Assumption \bref{assn:dense_closed_factn_sys} can be obtained as follows.  Take $\V := \X$, let $\B := \uX$, and let $(\E,\M) := (\Epi\X,\Mono\X)$.  Let $\TT$ be the double-dualization $\X$-monad for $\Omega_j$, and let $\Sigma := \Sigma_T$.  Then
\begin{enumerate}
\item an object of $\X$ is $\TT$-complete if and only if it is a $j$-sheaf;
\item an object of $\X$ is $\TT$-separated if and only if it is $j$-separated;
\item the $\TT$-closed $\M$-embeddings are exactly the $j$-closed monomorphisms;
\item the $\TT$-dense morphisms are exactly the $j$-dense morphisms.
\end{enumerate}
Also, given a Grothendieck quasitopos $\Y$, so that (up to equivalence of categories) $\Y$ is the category of $K$-separated $J$-sheaves on a \textit{bisite} $(\C,J,K)$, we show in \bref{par:gr_qtoposes} that data satisfying Assumption \bref{assn:dense_closed_factn_sys} can be obtained by considering the associated Lawvere-Tierney topologies $j,k$ on the presheaf topos $\X := [\C^\op,\Set]$, taking $\TT$ to be the double-dualization $\X$-monad for $\Omega_k$, and letting $\Sigma$ consist of all morphisms inverted by the double-dualization functor for $\Omega_j$, so that $\uX_{(\TT,\Sigma)} = \Y$.
\end{ExaSub}

\begin{PropSub}\label{thm:props_sigma_cl_dens}\emptybox
\begin{enumerate}
\item Every morphism $e \in \E$ is $\Sigma$-dense.
\item Every morphism in $\Sigma$ is $\Sigma$-dense.
\item If a composite $B_1 \xrightarrow{f} B_2 \xrightarrow{g} B_3$ is $\Sigma$-dense, then its second factor $g$ is $\Sigma$-dense.
\item If a composite $B_1 \xrightarrow{f} B_2 \xrightarrow{g} B_3$ and its second factor $g$ are both $\Sigma$-closed $\M$-embeddings, then the first factor $f$ is a $\Sigma$-closed $\M$-embedding.
\item Every $\Sigma$-dense $\Sigma$-closed $\M$-embedding is an isomorphism.
\item $\SigmaClEmb$ is closed under composition, cotensors, arbitrary $\V$-intersections, and $\V$-pullbacks along arbitrary morphisms in $\B$.
\item $\SigmaDense$ is closed under composition, tensors, arbitrary $\V$-cofibre-coproducts, and $\V$-pushouts along arbitrary morphisms in $\B$.
\end{enumerate}
\end{PropSub}
\begin{proof}
1.  $e$ is $\V$-orthogonal to every $\M$-embedding and hence to every $\Sigma$-closed $\M$-embedding.
2.  $\ClEmbWrt{\Sigma} \subs \Sigma^{\downarrow_\V}$, so $\Sigma \subs \Sigma^{\downarrow_\V\uparrow_\V} \subs \ClEmbWrt{\Sigma}^{\uparrow_\V} = \DenseWrt{\Sigma}$.
3-7 follow from \bref{par:stab_props_prefactn}.
\end{proof}

\begin{ParSub} \label{par:sigma_closure}
Since $(\SigmaDense,\SigmaClEmb)$ is a $\V$-factorization-system and hence an ordinary factorization system (by \cite[5.3]{Lu:EnrFactnSys}), it determines a weakly hereditary idempotent closure operator $\overline{(-)}$ on $\M$ in $\B$ \pbref{par:cl_op_assoc_to_factn_sys}.  For each $\M$-embedding $m$, we call $\overline{m}$ the \textit{$\Sigma$-closure} of $m$; in the case that $\Sigma = \Sigma_T$, we call $\overline{m}$ the \textit{$\TT$-closure} of $m$.
\end{ParSub}

\begin{PropSub} \label{thm:compl_vs_clemb}
Let $m:B' \rightarrowtail B$ be an $\M$-embedding in $\B$, and suppose $B$ is $\Sigma$-complete.
Then
$$\text{$B'$ is $\Sigma$-complete}\;\;\Leftrightarrow\;\;\text{$m$ is $\Sigma$-closed}\;.$$
\end{PropSub}
\begin{proof}
For each $h:B_1 \rightarrow B_2$ in $\Sigma$, we have a commutative square
$$
\xymatrix{
\B(B_2,B') \ar[r]^{\B(h,B')} \ar[d]_{\B(B_2,m)} & \B(B_1,B') \ar[d]^{\B(B_1,m)} \\
\B(B_2,B) \ar[r]_{\B(h,B)}                      & \B(B_1,B) 
}
$$
in which $\B(h,B)$ is iso, so the square is a pullback if and only if $\B(h,B')$ is iso. 
\end{proof}

For the remainder of this section, let us fix a $\V$-adjunction $F \dashv G:\C \rightarrow \B$ inducing $\TT$; for example, one can take $F \dashv G$ to be the Kleisli $\V$-adjunction.

\begin{PropSub} \label{thm:tc_complete}
Let $C$ be an object of $\C$.  Then $GC$ is $\Sigma$-complete.
\end{PropSub}
\begin{proof}
For each $h:B_1 \rightarrow B_2$ in $\Sigma$, we have a commutative square
$$
\xymatrix{
\B(B_2,GC) \ar[r]^{\B(h,GC)} \ar[d]_\wr & \B(B_1,GC) \ar[d]^{\wr} \\
\C(FB_2,C) \ar[r]_{\C(Fh,C)}      & \C(FB_1,C) 
}
$$
whose left and right sides are isomorphisms, and since $h \in \Sigma \subs \Sigma_T = \Sigma_F$ by \bref{thm:all_ladjs_ind_given_monad_inv_same_morphs}, the bottom side is iso, so the top side is iso.
\end{proof}

\begin{PropSub} \label{thm:sepobj_compl_iff_eta_closed}
Let $B \in \B$ be $\TT$-separated.  Then
$$\text{$B$ is $\Sigma$-complete}\;\;\Leftrightarrow\;\;\text{$\eta_B:B \rightarrowtail TB$ is $\Sigma$-closed}\;.$$
\end{PropSub}
\begin{proof}
Since $TB = GFB$ is $\Sigma$-complete by \bref{thm:tc_complete} and $\eta_B$ is an $\M$-embedding, this follows from \bref{thm:compl_vs_clemb}.
\end{proof}

\begin{PropSub} \label{thm:sepcompl_iff_exists_clemb}
An object $B \in \B$ is $\TT$-separated (resp. $\Sigma$-complete and $\TT$-separated) iff there exists an $\M$-embedding (resp. $\Sigma$-closed $\M$-embedding) \mbox{$m:B \rightarrowtail GC$} for some $C \in \C$.
\end{PropSub}
\begin{proof}
If $B$ is $\TT$-separated, then $\eta_B:B \rightarrow GFB$ is an $\M$-embedding; if $B$ is also $\Sigma$-complete, then by \bref{thm:sepobj_compl_iff_eta_closed}, $\eta_B$ is $\Sigma$-closed.  Conversely, if $m:B \rightarrowtail GC$ is an $\M$-embedding then we have a commutative triangle
$$
\xymatrix{
B \ar[r]^{\eta_B} \ar@{ >->}[rd]_m & GFB \ar@{-->}[d]^{Gm^\sharp} \\
                                  & GC                
}
$$
for a unique morphism $m^\sharp$ in $\C$, so $\eta_B \in \M$ by \bref{par:stab_props_prefactn} (since $\E \subseteq \Epi_\V\B$), so $B$ is $\TT$-separated.  If the given embedding $m$ is also $\Sigma$-closed, then since $GC$ is $\Sigma$-complete by \bref{thm:tc_complete}, we deduce by \bref{thm:compl_vs_clemb} that $B$ is $\Sigma$-complete.
\end{proof}

\begin{CorSub}\label{thm:carriers_of_free_talgs_sep_compl}
For each $B \in \B$, $TB$ is $\Sigma$-complete and $\TT$-separated. 
\end{CorSub}
\begin{proof}
$1_{TB}:TB \rightarrow TB = GFB$ is a $\Sigma$-closed $\M$-embedding.
\end{proof}

\begin{DefSub} \label{def:completion}
For each $B \in \B$, let 
$$
\xymatrix{
B \ar[rr]^{\eta_B} \ar[dr]_{\rho_B} &                        & TB \\
                                    & KB \ar@{ >->}[ur]_{\iota_B} & 
}
$$
be the $(\SigmaDense,\SigmaClEmb)$-factorization of $\eta_B$.
\end{DefSub}

\begin{PropSub} \label{thm:kb_sep_compl}
For each $B \in \B$, $KB$ is $\Sigma$-complete and $\TT$-separated.
\end{PropSub}
\begin{proof}
Since $TB = GFB$, this follows from \bref{thm:sepcompl_iff_exists_clemb}.
\end{proof}

The following lemma was inspired by an idea employed in the proof of 3.3 of \cite{CHK} in the non-enriched context.

\begin{LemSub} \label{thm:dense_and_sent_to_section_implies_iso}
Let $f:B_1 \rightarrow B_2$ be a $\Sigma$-dense morphism for which $Ff:FB_1 \rightarrow FB_2$ is a section.  Then $Ff$ is iso.
\end{LemSub}
\begin{proof}
The periphery of the following diagram commutes.
$$
\xymatrix{
B_1 \ar[r]^{\eta_{B_1}} \ar[d]_f        & GFB_1 \ar[d]^{GFf} \\
B_2 \ar[r]_{\eta_{B_2}} \ar@{-->}[ur]^k & GFB_2
}
$$
Also, $GFf$ is a section and hence is an $\M$-embedding (by \bref{par:vproper_prefactn}).  Further, by \bref{thm:charns_of_sigma_p}, we have that $GFf \in \Sigma_F^{\downarrow_\V} \subs \Sigma^{\downarrow_\V}$ (using the fact that $\Sigma \subs \Sigma_T = \Sigma_F$), so $GFf$ is a $\Sigma$-closed $\M$-embedding.  Hence, since $f$ is $\Sigma$-dense, there is a unique morphism $k$ making the above diagram commute.  In particular, $GFf \cdot k = \eta_{B_2}$; taking the transposes of both sides of this equation, with respect to the adjunction $F \dashv G$, we find that $Ff \cdot k^\sharp = 1_{FB_2}$ where $k^\sharp:FB_2 \rightarrow FB_1$ is the transpose of $k$.  Hence $Ff$ is a split epi and hence, being also a split mono, is iso.
\end{proof}

\begin{PropSub} \label{thm:rho_inverted_by_s}
For each $B \in \B$, $\rho_B:B \rightarrow KB$ is inverted by $F$.
\end{PropSub}
\begin{proof}
Taking the transposes of each side of the equation
$$\left(B \xrightarrow{\rho_B} KB \xrightarrow{\iota_B} GFB\right) = \eta_B$$
under the adjunction $F \nsststile{\varepsilon}{\eta} G$, we obtain
$$\left(FB \xrightarrow{F\rho_B} FKB \xrightarrow{F\iota_B} FGFB \xrightarrow{\varepsilon_{FB}} FB\right) = 1_{FB}\;,$$
so $F\rho_B$ is a section, so since $\rho_B$ is $\Sigma$-dense, \bref{thm:dense_and_sent_to_section_implies_iso} applies, and we deduce that $F\rho_B$ is iso.
\end{proof}

\begin{PropSub} \label{thm:rho_orth_to_each_sep_compl_ob}
Let $B,B' \in \B$ and suppose $B'$ is $\TT$-separated and $\Sigma$-complete.  Then $\rho_B \bot_\V B'$.
\end{PropSub}
\begin{proof}
We have a commutative diagram as follows.
$$
\xymatrix{
\B(KB,B') \ar[rr]^{\B(\rho_B,B')} \ar[d]_{\B(KB,\eta_{B'})} & & \B(B,B') \ar[d]^{\B(B,\eta_{B'})} \\
\B(KB,GFB') \ar[rr]^{\B(\rho_B,GFB')} \ar[d]^\wr           & & \B(B,GFB') \ar[d]^\wr \\
\C(FKB,FB') \ar[rr]^{\C(F\rho_B,FB')}                       & & \C(FB,FB')
}
$$
Since $B'$ is $\TT$-separated and $\Sigma$-complete, we have by \bref{thm:sepobj_compl_iff_eta_closed} that $\eta_{B'}$ is a $\Sigma$-closed $\M$-embedding, so since $\rho_B$ is $\Sigma$-dense, $\rho_B \downarrow_\V \eta_{B'}$, so the upper square is a pullback.  Also, $F\rho_B$ is iso by \bref{thm:rho_inverted_by_s}, so the left, bottom, and right sides of the lower square are iso.  Therefore $\B(\rho_B,GFB')$ is iso and hence its pullback $\B(\rho_B,B')$ is iso.
\end{proof}

\begin{ThmSub}\label{thm:refl_orth}
Let $\TT$ be a $\V$-monad on a $\V$-category $\B$ equipped with a $\V$-proper $\V$-prefactorization-system $(\E,\M)$, and let $\Sigma \subs \Sigma_T$.  Suppose that every morphism in $\B$ factors as a $\Sigma$-dense morphism followed by a $\Sigma$-closed $\M$-embedding.  Then the morphisms $\rho_B:B \rightarrow KB$ ($B \in \B$) of \bref{def:completion} exhibit the $\V$-category $\B_{(\TT,\Sigma)}$ of $\Sigma$-complete $\TT$-separated objects \pbref{def:compl_sep_closed_dense} as a $\V$-reflective-subcategory of $\B$.  
\end{ThmSub}
\begin{proof}
For each $B \in \B$ we have by \bref{thm:kb_sep_compl} that $KB$ lies in $\B_{(\TT,\Sigma)}$, and for each $B' \in \B_{(\TT,\Sigma)}$,
$$\B_{(\TT,\Sigma)}(KB,B') = \B(KB,B') \xrightarrow{\B(\rho_B,B')} \B(B,B')$$
is iso by \bref{thm:rho_orth_to_each_sep_compl_ob}.  The result follows by \bref{prop:adjn_via_radj_and_unit}.
\end{proof}

\begin{RemSub}
The hypothesis in \bref{thm:refl_orth} that $(\SigmaDense,\SigmaClEmb)$-factorizations exist (i.e., Assumption \bref{assn:dense_closed_factn_sys}) is satisfied as soon as $\B$ is \mbox{\textit{$\M$-subobject-complete}} \pbref{par:m_subobj_compl}.
\end{RemSub}

\begin{DefSub}
Given data satisfying the hypotheses of \bref{thm:refl_orth}, we call the idempotent $\V$-monad $\tTT_\Sigma$ on $\B$ induced by the resulting $\V$-reflection $K \nsststile{}{\rho} J:\B_{(\TT,\Sigma)} \hookrightarrow \B$ the \textit{$\TT$-separated $\Sigma$-completion $\V$-monad}.
\end{DefSub}

\begin{PropSub}\label{thm:suppl_results_on_tsepsigmacompl_mnd}
Suppose given data satisfying the hypotheses of \bref{thm:refl_orth}.
\begin{enumerate}
\item There is a unique morphism of $\V$-monads $\iota:\tTT_\Sigma \rightarrow \TT$.
\item Each component of $\iota$ is a $\Sigma$-closed $\M$-embedding.
\item Each component of the unit $\rho:1_\B \rightarrow \tT_\Sigma$ of $\tTT_\Sigma$ is a $\Sigma$-dense morphism.
\item Each $\V$-adjunction inducing $\TT$ factors through the $\V$-reflection $K \dashv J:\B_{(\TT,\Sigma)} \hookrightarrow \B$ determined by $\tTT_\Sigma$.
\end{enumerate}
\end{PropSub}
\begin{proof}
1 and 4 follow from \bref{thm:carriers_of_free_talgs_sep_compl} and \bref{thm:charns_adj_factn_via_mnd}.  By \bref{thm:charn_comps_morph_on_idm_mnd} and \bref{def:completion}, the components of the resulting morphism of $\V$-monads are necessarily the $\Sigma$-closed $\M$-embeddings $\iota_B$ given in \bref{def:completion}, so 2 holds.  3 is immediate.
\end{proof}

\begin{CorSub} \label{thm:sep_refl}
Let $\TT$ be a $\V$-monad on a $\V$-category $\B$ equipped with a $\V$-proper $\V$-factorization-system $(\E,\M)$.  Then the full sub-$\V$-category of $\B$ consisting of the $\TT$-separated objects is $\V$-reflective in $\B$.
\end{CorSub}
\begin{proof}
Taking $\Sigma = \emptyset$, the objects of $\B_{(\TT,\Sigma)}$ are exactly the $\TT$-separated objects of $\B$, and the hypotheses of \bref{thm:refl_orth} are satisfied since $\ClEmbWrt{\emptyset} = \M$ and hence $\DenseWrt{\emptyset} = \E$.
\end{proof}

\begin{ThmSub} \label{thm:sigma_s_complete_are_sep_and_refl}
Let $\TT$ be a $\V$-monad on a $\V$-category $\B$ equipped with a $\V$-proper $\V$-factorization-system $(\E,\M)$, and suppose that every morphism in $\B$ factors as a $\TT$-dense morphism followed by a $\TT$-closed $\M$-embedding.
\begin{enumerate}
\item The idempotent $\V$-core $\tTT$ of $\TT$ exists.
\item The $\V$-reflective-subcategory of $\B$ determined by $\tTT$ consists of the $\TT$-complete objects.
\item Every $\TT$-complete object of $\B$ is $\TT$-separated.
\item Each component of the unique morphism of $\V$-monads $\iota:\tTT \rightarrow \TT$ is a $\TT$-closed $\M$-embedding.
\item Each component of the unit $\rho:1_\B \rightarrow \tT$ of $\tTT$ is a $\TT$-dense morphism.
\item Every $\V$-adjunction $F \dashv G:\C \rightarrow \B$ inducing $\TT$ factors through the $\V$-reflection $K \dashv J:\B_{\Sigma_T} \hookrightarrow \B$ determined by $\tTT$, in such a way that the induced left $\V$-adjoint $F':\B_{\Sigma_T} \rightarrow \C$ \pbref{def:adj_factors_through_refl} is conservative.
\end{enumerate}
\end{ThmSub}
\begin{proof}
3. By \bref{thm:sep_refl} we know that the full sub-$\V$-category $\B_{(\TT,\emptyset)}$ of $\B$ consisting of the $\TT$-separated objects is $\V$-reflective in $\B$, and we will denote the components of the unit of the associated $\V$-reflection by $\sigma_B:B \rightarrow LB$ ($B \in \B$).  Hence, $\sigma_B$ is gotten as the morphism $\rho_B$ of \bref{def:completion} in the case that $\Sigma = \emptyset$.  By \bref{thm:rho_inverted_by_s} we know that each such component $\sigma_B$ is inverted by $F$ --- i.e. $\sigma_B \in \Sigma_F = \Sigma_T$.  Hence, given any $\TT$-complete object $B' \in \B_{\Sigma_T}$, we have that $\sigma_B \bot_\V B'$ for every $B \in \B$, so by \bref{thm:refl_subcats_are_orth_subcats}, $B' \in \B_{(\TT,\emptyset)}$.  

By 3 we know that $\B_{\Sigma_T} = \B_{(\TT,\Sigma_T)}$, and by \bref{thm:refl_orth} we deduce that $\B_{(\TT,\Sigma_T)}$ is a $\V$-reflective-subcategory of $\B$.  Hence 1, 2, and 6 follow immediately from \bref{thm:charn_idm_approx}.

4 and 5 follow from \bref{thm:suppl_results_on_tsepsigmacompl_mnd}.
\end{proof}

\begin{CorSub}\label{thm:msubobjcompl_impl_idm_approx_exists}
Let $\TT$ be a $\V$-monad on an $\M$-subobject-complete $\V$-category $\B$, where $(\E,\M)$ is a $\V$-proper \mbox{$\V$-prefactorization-system} on $\B$.  Then the idempotent $\V$-core $\tTT$ of $\TT$ exists.
\end{CorSub}
\begin{proof}
The hypotheses of \bref{thm:sigma_s_complete_are_sep_and_refl} are satisfied (\bref{par:m_subobj_compl}, \bref{exa:msubobjcpl_impl_factn_assn}).
\end{proof}

In view of \bref{thm:sigma_s_complete_are_sep_and_refl}, we shall extend the notation and terminology of \bref{def:compl_sep_closed_dense} as follows:

\begin{DefSub}\label{def:tcompletion}
Let data satisfying the hypotheses of \bref{thm:sigma_s_complete_are_sep_and_refl} be given.
\begin{enumerate}
\item We call $\tTT$ the \textit{$\TT$-completion $\V$-monad}.
\item For each object $B \in \B$, we call $\tT B = KB$ the \textit{$\TT$-completion} of $B$.
\item We denote by $\B_{(\TT)} := \B_{\Sigma_T} = \B_{(\TT,\Sigma_T)}$ the $\V$-reflective-subcategory of $\B$ consisting of the $\TT$-complete objects.
\end{enumerate}
\end{DefSub}

\section{Example: Closure and density in normed spaces}\label{sec:cl_dens_nvs}

We saw in \bref{thm:compl_nvs} that the completion $\V$-monad on the category $\V$ of normed (resp. seminormed) vector spaces is the idempotent $\V$-core of the double-dualization $\V$-monad on $\uV$.  Further, we shall see that the familiar construction of the completion of a normed or seminormed vector space $V$ as the closure of $V$ in $V^{**}$ is an instance of the general procedure given in \bref{thm:sigma_s_complete_are_sep_and_refl} for forming the idempotent $\V$-core $\tTT$ of a $\V$-monad $\TT$.  Indeed, we will show that the notions of $\TT$-closure and $\TT$-density in this example coincide with the familiar notions of closure and density.

Again as in \bref{sec:compl_nvs}, let $\V$ be either $\SNorm_1$ or $\Norm_1$, and let $\TT$ denote the double-dualization $\V$-monad on $\uV$.  Further, let $\E$ denote the class of all surjective morphisms in $\V$, and let $\M$ denote the class of all isometric embeddings in $\V$.

\begin{LemSub}\label{thm:surj_isoemb_vfs}
$(\E,\M)$ is a $\V$-proper $\V$-factorization-system on the $\V$-category $\uV$ of normed (resp. seminormed) vector spaces.
\end{LemSub}
\begin{proof}
Each surjective morphism (resp. each isometric embedding) is clearly an epimorphism (resp. a monomorphism) in $\V$ and hence, by \cite[2.4]{Lu:EnrFactnSys}, is a $\V$-epi (resp. a $\V$-mono) in $\uV$.  Since $\uV$ is a cotensored $\V$-category, it suffices by \cite[5.7]{Lu:EnrFactnSys} to show that $(\E,\M)$ is an ordinary factorization system on $\V$ and that $\M$ is closed under cotensors in $\uV$.  Clearly every morphism in $\V$ factors as a morphism in $\E$ followed by a morphism in $\M$.  Moreover, it is easy to check that each morphism $e \in \E$ is orthogonal to each morphism $m \in \M$.  Hence, since $\E$ and $\M$ are also closed under composition with isomorphisms, we deduce that $(\E,\M)$ is a factorization system on $\V$ (e.g., by \cite[5.2]{Lu:EnrFactnSys}).  Further $\M$ is closed under cotensors in $\V$, since for any isometric embedding $m:W_1 \rightarrow W_2$ in $\V$, it is readily verified that the induced morphism $\uV(V,m):\uV(V,W_1) \rightarrow \uV(V,W_2)$ (given by composing with $m$) is injective and isometric.
\end{proof}

\begin{ThmSub}\label{thm:dens_cl_nvs}
With respect to the double-dualization $\V$-monad $\TT$ on the category $\V$ of normed (resp. seminormed) vector spaces and the $\V$-factorization-system $(\E,\M)$, the following hold:
\begin{enumerate}
\item A morphism in $\V$ is $\TT$-dense if and only if it is dense.
\item A morphism in $\V$ is a $\TT$-closed $\M$-embedding if and only if it is a closed isometric embedding.
\item Every morphism in $\V$ factors as a $\TT$-dense morphism followed by a $\TT$-closed $\M$-embedding.
\item The $\TT$-closure of a subspace $V \hookrightarrow W$ of a normed (resp. seminormed) vector space $W$ is the usual closure $\overline{V} \hookrightarrow W$ of $V$ in $W$.
\end{enumerate}
\end{ThmSub}
\begin{proof}
One sees immediately that the familiar process of taking the closure of a subspace defines a weakly hereditary idempotent closure operator $\overline{(-)}$ on $\M$ in $\V$ \pbref{par:idm_cl_op}.  By \bref{par:factn_sys_from_cl_op}, $\overline{(-)}$ determines an associated factorization system $(\Dense,\ClEmb)$, consisting of the dense morphisms and closed $\M$-embeddings, respectively.

Recall from \bref{thm:compl_nvs} that the class $\Sigma_T$ of all morphisms inverted by $T = (-)^{**}$ consists of exactly the dense isometric morphisms.  By definition, an $\M$-embedding $m$ is \mbox{$\TT$-closed} iff $m \in \Sigma_T^{\downarrow_\V}$.  But since $(\Sigma_T,\Sigma_T^{\bot_\V})$ is a $\V$-orthogonal-pair in $\uV$ \pbref{thm:morphs_inv_by_t_vorth_pair}, we deduce by \bref{thm:clos_props_orth_pairs} that $\Sigma_T$ is closed under tensors in $\uV$, so by \cite[5.4]{Lu:EnrFactnSys}, $\Sigma_T^{\downarrow_\V} = \Sigma_T^{\downarrow}$, whence $\ClEmbWrt{\TT} = \Sigma_T^\downarrow \cap \M$.

To prove 2, first observe that since $\Sigma_T \subseteq \Dense$, $\ClEmb = \Dense^\downarrow \subseteq \Sigma_T^\downarrow$ and hence $\ClEmb \subseteq \ClEmbWrt{\TT}$.  Also, by \bref{thm:clemb_via_densemb} we know that $\ClEmb = \DenseEmb^\downarrow \cap \M$, where $\DenseEmb$ is the set of all dense isometric embeddings, so that since $\DenseEmb \subseteq \Sigma_T$, $\Sigma_T^\downarrow \subseteq \DenseEmb^\downarrow$ and hence $\ClEmbWrt{\TT} \subseteq \ClEmb$.

Since both $\Sigma_T^{\downarrow_\V}$ and $\M$ are closed under cotensors in $\uV$ (by \bref{par:stab_props_prefactn} and \bref{thm:morphs_inv_by_t_vorth_pair}), their intersection $\ClEmbWrt{\TT} = \ClEmb$ is closed under cotensors in $\uV$, so the factorization system $(\Dense,\ClEmb)$ is a $\V$-factorization-system, by \cite[5.7]{Lu:EnrFactnSys}.  Hence 1 and 3 follow, since $\Dense = \ClEmb^{\uparrow_\V} = (\ClEmbWrt{\TT})^{\uparrow_\V} = \DenseWrt{\TT}$.  Further, 4 follows as well, since the $\TT$-closure
operator is by definition the closure operator determined by $(\DenseWrt{\TT},\ClEmbWrt{\TT}) = (\Dense,\ClEmb)$ \pbref{par:cl_op_assoc_to_factn_sys}, which coincides with $\overline{(-)}$.
\end{proof}

\begin{RemSub}
By \bref{thm:surj_isoemb_vfs} and \bref{thm:dens_cl_nvs} 3, the data $\TT$, $(\E,\M)$ of \bref{thm:dens_cl_nvs} satisfy the hypotheses of \bref{thm:sigma_s_complete_are_sep_and_refl}, and the latter theorem entails that the idempotent $\V$-core $\tTT$ of $\TT$ can be formed as follows:  For each object $V \in \V$, $\iota_V:\tT V \hookrightarrow TV$ is the $\TT$-closure of the image of $\eta_V:V \rightarrow TV$.  But by \bref{thm:dens_cl_nvs}, this is simply the usual closure of the image of the canonical map $V \rightarrow V^{**}$; i.e., $\tT V$ is the usual completion of $V$.
\end{RemSub}

\section{Example: Sheafification, closure, and density}\label{sec:shfn_cl_dens}

Let $j$ be a Lawvere-Tierney topology on an (elementary) topos $\X$, let $\Omega_j$ be the associated retract of the subobject classifier $\Omega$, let $\TT$ be the double-dualization $\X$-monad on $\uX$ determined by $\Omega_j$ \pbref{exa:dbl_dln}, and let $(\E,\M) = (\Epi\X,\Mono\X)$.  We now show by means of \bref{thm:sigma_s_complete_are_sep_and_refl} that the idempotent $\X$-core of $\TT$ exists and is given by \textit{$j$-sheafification}, and that moreover, the notions of $\TT$-density, $\TT$-closure, $\TT$-separated object, and $\TT$-complete object coincide with the familiar notions of $j$-density, $j$-closure, $j$-separated object, and $j$-sheaf.

\begin{ParSub}\label{par:lt_top_defs}
The \textit{universal closure operator} $\overline{(-)}$ determined by $j$ is, in particular, a weakly hereditary idempotent closure operator on $\M$ in $\X$; indeed, see \cite[A4.3.2, A4.3.3(ii)]{Joh:Ele}.  This closure operator determines an associated factorization system $(\DenseWrt{j},\ClEmbWrt{j})$ \pbref{par:factn_sys_from_cl_op}.  The elements of $\DenseWrt{j}$ are called \textit{$j$-dense morphisms} and those of $\ClEmbWrt{j}$ \textit{\mbox{$j$-closed} monomorphisms}.  An object $Y \in \X$ is said to be a \textit{$j$-separated} (resp., a \textit{$j$-sheaf}) if for every $j$-dense monomorphism $d:D \rightarrowtail X$, the mapping $\X(d,Y):\X(X,Y) \rightarrow \X(D,Y)$ is a injective (resp. bijective).  Hence, letting $\DenseEmbWrt{j}$ be the class of all $j$-dense monomorphisms, the class of all $j$-sheaves is $\Shv(\X,j) = (\DenseEmbWrt{j})^{\bot}$.  Also, by \bref{thm:clemb_via_densemb}, $\ClEmbWrt{j} = (\DenseEmbWrt{j})^\downarrow \cap \M$.
\end{ParSub}

\begin{LemSub}\label{thm:charn_jsep_jshv}
An object $X \in \X$ is $j$-separated (resp. a $j$-sheaf) if and only if there exists a monomorphism (resp. a $j$-closed mono) $X \rightarrowtail \uX(Y,\Omega_j)$ for some $Y \in \X$.
\end{LemSub}
\begin{proof}
If an object $X \in \X$ is $j$-separated, then by \cite[V.3.4]{MacMoe} there is a monomorphism $X \rightarrowtail \uX(X,\Omega_j)$.  If $X$ is moreover a $j$-sheaf, then since $\uX(X,\Omega_j)$ is a $j$-sheaf by \cite[3.27, 3.24]{Joh:TopTh} we deduce by \cite[3.26]{Joh:TopTh} that the above monomorphism $X \rightarrowtail \uX(X,\Omega_j)$ is $j$-closed.  Conversely, if there is a monomorphism $m:X \rightarrowtail \uX(Y,\Omega_j)$ for some $Y \in \Y$, then by \cite[3.27, 3.24, 3.26]{Joh:TopTh}, $X$ is $j$-separated.  If the given monomorphism $m$ is $j$-closed, then we deduce by \cite[3.27, 3.24, 3.26]{Joh:TopTh} that $X$ is a $j$-sheaf.
\end{proof}

\begin{LemSub}\label{thm:jdense_mono_closed_under_tensors}
The class $\DenseEmbWrt{j}$ of all $j$-dense monomorphisms is closed under tensors in $\uX$.
\end{LemSub}
\begin{proof}
Let $d:D \rightarrow X$ be a $j$-dense monomorphism, and let $Y \in \X$.  Then $Y \times d:Y \times D \rightarrow Y \times X$ is a pullback of $d$ (along the projection $Y \times X \rightarrow X$), and $\overline{(-)}$ commutes (up to isomorphism) with pullback (e.g., by \cite[4.3.2]{Joh:Ele}), so the object $\overline{Y \times d}:\overline{Y \times D} \rightarrow Y \times X$ of the category $\Sub_\M(Y \times X)$ is isomorphic to the pullback $Y \times \overline{d}:Y \times \overline{D} \rightarrow Y \times X$ of $\overline{d}:\overline{D} \rightarrow X$.  But $\overline{d} \cong 1_X$ in $\Sub_\M(X)$ and hence $\overline{Y \times d} \cong Y \times \overline{d} \cong Y \times 1_X = 1_{Y \times X}$ in $\Sub_\M(Y \times X)$, so $Y \times d$ is a $j$-dense monomorphism.
\end{proof}

\begin{CorSub}\label{thm:charns_shvs_and_jclemb_via_enr_orth}
$\Shv(\X,j) = (\DenseEmbWrt{j})^{\bot_\X}$, and $\ClEmbWrt{j} = (\DenseEmbWrt{j})^{\downarrow_\X} \cap \M$.
\end{CorSub}
\begin{proof}
It was remarked in \bref{par:lt_top_defs} that
$$\Shv(\X,j) = (\DenseEmbWrt{j})^{\bot},\;\;\;\;\ClEmbWrt{j} = (\DenseEmbWrt{j})^{\downarrow} \cap \M,$$
so in view of \bref{thm:jdense_mono_closed_under_tensors} the needed equations follow from \bref{thm:case_where_ord_orth_impl_enr} and \cite[5.4]{Lu:EnrFactnSys}.
\end{proof}

\begin{LemSub}\label{thm:jdense_jclemb_enrfs}
$(\DenseWrt{j},\ClEmbWrt{j})$ is an $\X$-factorization-system on $\uX$.
\end{LemSub}
\begin{proof}
By \cite[2.4]{Lu:EnrFactnSys}, monomorphisms in $\X$ are the same as $\X$-enriched monomorphisms in $\uX$, so $\M = \Mono_\X\uX$ and hence by \cite[2.11]{Lu:EnrFactnSys}, $\M$ is closed under cotensors in $\uX$.  Since $(\DenseEmbWrt{j}^{\downarrow_\X\uparrow_\X},\DenseEmbWrt{j}^{\downarrow_\X})$ is an $\X$-prefactorization-system on $\uX$, $\DenseEmbWrt{j}^{\downarrow_\X}$ is also closed under cotensors in $\uX$ by \bref{par:stab_props_prefactn}.  Hence, since $\ClEmbWrt{j}$ is the intersection of the classes $(\DenseEmbWrt{j})^{\downarrow_\X}$ and $\M$ by \bref{thm:charns_shvs_and_jclemb_via_enr_orth}, $\ClEmbWrt{j}$ is closed under cotensors in $\uX$, so by \cite[5.7]{Lu:EnrFactnSys}, the factorization system $(\DenseWrt{j},\ClEmbWrt{j})$ is an $\X$-factorization-system on $\uX$.
\end{proof}

\begin{LemSub}\label{thm:jdense_emb_containedin_sigmat}
$\DenseEmbWrt{j} \subseteq \Sigma_T$, where $\Sigma_T$ is the class of all morphisms inverted by the endofunctor $T$.
\end{LemSub}
\begin{proof}
$\Omega_j$ is a $j$-sheaf (e.g., by \cite[3.27]{Joh:TopTh}), so by \bref{thm:charns_shvs_and_jclemb_via_enr_orth}, $\Omega_j \in (\DenseEmbWrt{j})^{\bot_\X}$.  Hence each morphism $d \in \DenseEmbWrt{j}$ is $\X$-orthogonal to $\Omega_j$ and therefore is inverted by the dualization functor $\uX(-,\Omega_j):\X \rightarrow \X^\op$ and hence by the double-dualization functor $T$ for $\Omega_j$.
\end{proof}

\begin{LemSub}\label{thm:sigmat_containedin_jdense}
$\Sigma_T \subseteq \DenseWrt{j}$.
\end{LemSub}
\begin{proof}
Suppose $h:X_1 \rightarrow X_2$ lies in $\Sigma_T$.  Since the monad $\TT$ is induced by the dualization adjunction $\uX(-,\Omega_j) \dashv \uX(-,\Omega_j):\X^\op \rightarrow \X$ \pbref{exa:dbl_dln}, we know by \bref{thm:all_ladjs_ind_given_monad_inv_same_morphs} that $\Sigma_T = \Sigma_{\uX(-,\Omega_j)}$, so $\uX(h,\Omega_j):\uX(X_2,\Omega_j) \rightarrow \uX(X_1,\Omega_j)$ is an isomorphism.  But $\Omega_j$ classifies $j$-closed subobjects (\cite[V.2.2]{MacMoe}), so $h^{-1}:\Sub_\M(X_2) \rightarrow \Sub_\M(X_2)$ induces a bijection between the $j$-closed subobjects of $X_2$ and $X_1$, respectively.  Letting $h(X_1)$ denote the image of $h$, considered as subobject of $X_2$, we observe that its $j$-closure $\overline{h(X_1)}$ is a $j$-closed subobject of $X_2$ and has inverse image $h^{-1}(\overline{h(X_1)}) = X_1$.  But $X_2$ itself is also a $j$-closed subobject of $X_2$ with the same inverse image, $h^{-1}(X_2) = X_1$, so $\overline{h(X_1)} = X_2$ (as subobjects of $X_2$), showing that $h$ is $j$-dense.
\end{proof}

\begin{ThmSub}\label{thm:lt_top}
Given a Lawvere-Tierney topology $j$ on a topos $\X$, if we take $\TT$ to be the double-dualization $\X$-monad for $\Omega_j$ and let $(\E,\M)$ be the epi-mono factorization system on $\X$, then the following hold:
\begin{enumerate}
\item A morphism is $\TT$-dense if and only if it is $j$-dense.
\item A monomorphism is $\TT$-closed if and only if it is $j$-closed.
\item Every morphism in $\X$ factors as a $\TT$-dense morphism followed by a $\TT$-closed monomorphism.
\item The $\TT$-closure of a subobject is the same as its $j$-closure.
\item An object of $\X$ is $\TT$-complete if and only if it is a $j$-sheaf.
\item An object of $\X$ is $\TT$-separated if and only if it is $j$-separated.
\item The idempotent $\X$-core $\tTT$ of $\TT$ exists, and its associated $\X$-reflective-subcategory of $\uX$ consists of the $j$-sheaves.
\end{enumerate}
\end{ThmSub}
\begin{proof}
To prove 2, observe first that by \bref{thm:jdense_emb_containedin_sigmat} and \bref{thm:charns_shvs_and_jclemb_via_enr_orth},
$$\ClEmbWrt{\TT} = \Sigma_T^{\downarrow_\X} \cap \M \subseteq (\DenseEmbWrt{j})^{\downarrow_\X} \cap \M = \ClEmbWrt{j}.$$
Also, by \bref{thm:jdense_jclemb_enrfs} and \bref{thm:sigmat_containedin_jdense},
$$\ClEmbWrt{j} = (\DenseWrt{j})^{\downarrow_\X} \subseteq \Sigma_T^{\downarrow_\X}\;,$$
so $\ClEmbWrt{\TT} = \ClEmbWrt{j}$ and 2 is proved.  Now 1, 3, and 4 follow, since by \bref{thm:jdense_jclemb_enrfs} and \bref{rem:sigmadense_sigma_clemb_vprefs} we find that
$$\DenseWrt{j} = (\ClEmbWrt{j})^{\uparrow_\X} = (\ClEmbWrt{\TT})^{\uparrow_\X} = \DenseWrt{\TT}\;.$$

It now follows by \bref{thm:charn_jsep_jshv} that an object $X \in \X$ is $j$-separated (resp. a $j$-sheaf) if and only if there is an $\M$-embedding (resp. a $\TT$-closed $\M$-embedding) $X \rightarrowtail \uX(Y,\Omega_j)$ for some $Y \in \X$.  Hence, since $\TT$ is induced by the $\X$-adjunction $F \dashv G:\uX^\op \rightarrow \uX$ with $F = G = \uX(-,\Omega_j)$ \pbref{exa:dbl_dln}, we deduce 5 and 6 by \bref{thm:sepcompl_iff_exists_clemb}, using the fact that by \bref{thm:sigma_s_complete_are_sep_and_refl}, every $\TT$-complete object is $\TT$-separated.

In view of the above, 7 now follows from \bref{thm:sigma_s_complete_are_sep_and_refl}. 
\end{proof}

\begin{ParSub}\label{par:gr_qtoposes}
Let $\Y$ be a \textit{Grothendieck quasitopos}; equivalently, let $\Y$ be the category of \mbox{$K$-separated} $J$-sheaves on a small \textit{bisite} $(\C,J,K)$ (\cite[C2.2.13]{Joh:Ele}), so that $\C$ is a small category and $J,K$ are coverages on $\C$ with $J \subseteq K$.  Let $\X := [\C^\op,\Set]$ be the presheaf topos, and let $j,k$ be the Lawvere-Tierney topologies associated to $J,K$, respectively.  For each $i := j,k$, let $\TT_i$ be the double-dualization $\X$-monad on $\uX$ for $\Omega_i$, and let $\Sigma_i := \Sigma_{T_i}$ be the class of all morphisms inverted by $T_i$.  By \bref{thm:lt_top}, $\Y$ consists of the \mbox{$\TT_k$-separated} $\TT_j$-complete (equivalently, $\Sigma_j$-complete) objects of $\X$, i.e  $\Y = \uX_{(\TT_k,\Sigma_j)}$.  By \bref{thm:morphs_inv_by_t_vorth_pair}, $(\Sigma_i, \Sigma_i^{\bot_\X})$ is an $\X$-orthogonal-pair for each $i := j,k$, and by \bref{thm:lt_top} 5,
$$\Sigma_k^{\bot_\X} = \Shv(\C,K) \subseteq \Shv(\C,J) = \Sigma_j^{\bot_\X}\;,$$
so we deduce that $\Sigma_j \subseteq \Sigma_k$.  Applying \bref{thm:lt_top} 3 to $j$, we find that $\TT_k$ and $\Sigma_j$ satisfy the hypotheses of \bref{thm:refl_orth}, and we deduce that $\Y = \uX_{(\TT_k,\Sigma_j)}$ is an $\X$-enriched reflective subcategory of $\uX$.
\end{ParSub}

\bibliographystyle{amsplain}
\bibliography{bib}

\end{document}